\documentclass[11pt, twoside]{article}

\usepackage[english]{babel}

\usepackage{bbm}
\usepackage{amssymb}
\usepackage{amsfonts}
\usepackage{amsmath}
\usepackage{amsthm}
\usepackage{color}
\usepackage{mathrsfs}
\usepackage{txfonts}
\usepackage{bbm}
\usepackage{enumerate}
\usepackage{anysize}
\usepackage{indentfirst}

\usepackage{enumitem}

\newlist{Plist}{enumerate}{1}
\setlist[Plist]{
  label=\textup{(P\arabic*)},
  ref=\textup{(P\arabic*)},
  leftmargin=*,
  labelsep=0.5em
}

\newlist{Romanlist}{enumerate}{1}
\setlist[Romanlist]{
  label=\textup{(\Roman*)},
  ref=\textup{(\Roman*)},
  leftmargin=*,
  labelsep=0.5em
}

\usepackage[T1]{fontenc}
\usepackage[utf8]{inputenc}

\usepackage{latexsym}

\usepackage[colorlinks=true,
  linkcolor=blue,
  citecolor=red,
  urlcolor=magenta,
  ]{hyperref}

\allowdisplaybreaks

\newtheorem{theorem}{Theorem}[section]
\newtheorem{lemma}[theorem]{Lemma}
\newtheorem{corollary}[theorem]{Corollary}

\theoremstyle{definition}
\newtheorem{remark}[theorem]{Remark}

\newcounter{assum}

\newtheorem{question}{Question}

\renewcommand{\appendix}{\par
\setcounter{section}{0}%
\setcounter{subsection}{0}%
\setcounter{subsubsection}{0}%
\gdef\thesection{\@Alph\c@section}%
\gdef\thesubsection{\@Alph\c@section.\@arabic\c@subsection}%
\gdef\theHsection{\@Alph\c@section.}%
\gdef\theHsubsection{\@Alph\c@section.\@arabic\c@subsection}%
\csname appendixmore\endcsname
}

\numberwithin{equation}{section}

\begin{document}

\arraycolsep=1pt

\title{\bf\Large Parabolic BMO Spaces,
Muckenhoupt Weights, and
Reverse H\"older Classes with Time Lag:
Equivalence and Characterizations
\footnotetext{\hspace{-0.35cm} 2020
\emph{Mathematics Subject Classification}.
Primary 42B35; Secondary 42B37, 46E30.
\endgraf \emph{Key words and phrases}.
parabolic BMO space,
parabolic Muckenhoupt class,
parabolic reverse H\"older class.
\endgraf
This project is partially supported by the National Natural
Science Foundation of China (Grant Nos. 12431006 and 12371093), the Beijing
Natural Science Foundation (Grant No.1262011),
and the Fundamental Research Funds for the Central Universities
(Grant No. 2233300008).}}
\author{Weiyi Kong,
Dachun Yang and Wen Yuan}
\maketitle

\vspace{-0.8cm}

\begin{center}
\begin{minipage}{13cm}
{\small {\bf Abstract}\quad For any given time lag
$\gamma\in(0,1)$, we prove that the one-sided parabolic BMO space
$\mathrm{BMO}^+(\gamma)$ coincides with
the parabolic BMO space $\mathrm{PBMO}^-(\gamma)$ with equivalent norms,
the parabolic Muckenhoupt class $A_{\infty}^+(\gamma)$ defined via the
reverse Jensen inequality can be represented as the union of the parabolic
Muckenhoupt classes $A_r^+(\gamma)$ with $r\in[1,\infty)$, and
the parabolic reverse H\"older classes $\bigcup_{q\in(1,\infty]}
RH_q^+$ coincide with the parabolic Muckenhoupt classes
$\bigcup_{r\in[1,\infty)}A_r^+(\gamma)$, and hence give affirmative
answers to Questions 4.5 and 4.6 posed by Kinnunen and Saari
[Nonlinear Anal. 131 (2016)]. To show them, we establish the uniform
parabolic space-time shifting property for parabolic reverse H\"older
weights, and develop the one-sided stopping time argument
which yields a new parabolic John--Nirenberg inequality for
$\mathrm{BMO}^+(\gamma)$. As applications, we obtain John--Nirenberg and
exponential integrability characterizations of $\mathrm{BMO}^+(\gamma)$, prove
that $\mathrm{BMO}^+(\gamma)$ is independent of the positive time lag,
and identify its null space.
}
\end{minipage}
\end{center}

\vspace{0.2cm}


\vspace{0.2cm}

\section{Introduction}

Throughout this article, we work in $\mathbb{R}^{n+1}$ and,
unless otherwise specified, we take $\mathbb{R}^{n+1}$ as the default
underlying space. The space of functions with bounded mean oscillation, BMO,
plays a central role in harmonic analysis and PDEs (see, for example,
\cite{bd(cvpde-2017), bhh(jfa-2024), fs(acta-1972), h(rm-2016), ins(jia-2019),
jowz(ma-2025), jwyy(tams-2017), jn(cpam-1961), n(sm-1997), ns(jfa-2012)}).
Recall that the \emph{space} BMO is defined to be the set of all locally
integrable functions $f$ on $\mathbb{R}^{n}$
such that
$$\|f\|_{\mathop{\mathrm{BMO}}(\mathbb{R}^{n})}:=\sup_Q \frac{1}{|Q|}
\int_Q\left|f(x)-\frac{1}{|Q|}\int_Q f(y)\,dy\right|\,dx<\infty,$$
where the supremum is taken over all cubes $Q$ in $\mathbb{R}^{n}$
with edges parallel to coordinate axes. In the parabolic setting, parabolic
BMO spaces with time lag naturally arise in the study of the following
doubly nonlinear parabolic equation
on $\mathbb{R}^{n+1}$
\begin{align}\label{20260811.1443}
\frac{\partial}{\partial t}\left(|u|^{p-2}u\right)
-\mathrm{div}\left(|\nabla u|^{p-2}\nabla u\right)=0,
\end{align}
where $p\in(1,\infty)$ is \emph{always fixed} throughout this article
(see, for example, \cite{fg(pams-1985), kk(ma-2007), m(cpam-1964),
m(cpam-1964)2, t(cpam-1968)}). If $p=2$, then \eqref{20260811.1443} reduces to
the standard heat equation. Note that, if $u(x,t)$ is a solution to
\eqref{20260811.1443}, then, for any $\lambda\in(0,\infty)$, $u(\lambda x,
\lambda^p t)$ remains its solution. Due to this fundamental scaling behavior of
\eqref{20260811.1443}, the intrinsic geometry associated with
\eqref{20260811.1443} requires a non-standard temporal scaling of order
$p$. Thus, in the estimates concerning \eqref{20260811.1443}, classical Euclidean
cubes are naturally replaced by space-time parabolic rectangles. To formalize
the parabolic rectangle, let $(x,t):=(x_1,\dots,x_n,t)\in\mathbb{R}^{n+1}$ and
$L\in(0,\infty)$. A \emph{parabolic rectangle} $R(x,t,L)$ centered at $(x,t)$
with edge length $l(R):=L$ is defined by setting
\begin{align*}
R:=R(x,t,L):=Q(x,L)\times\left[t-L^p,t+L^p\right),
\end{align*}
where $Q(x,L):=\{y:=(y_1,\dots,y_n)\in\mathbb{R}^n:x_i-\frac{L}{2}\leq
y_i<x_i+\frac{L}{2}\mbox{\ for\ any\ }i\in\mathbb{N}\cap[1,n]\}$ is the
\emph{cube} in $\mathbb{R}^n$ centered at $x$ with edge length $L$. In
addition, for any given $\gamma\in[0,1)$, define $R^-(\gamma)$, $R^+(\gamma)$,
and $R^{++}(\gamma)$, respectively, by setting
\begin{align*}
R^-(\gamma):=Q(x,L)\times\left[t-L^p,t-\gamma L^p\right),\ \
R^+(\gamma):=Q(x,L)\times\left[t+\gamma L^p,t+L^p\right),
\end{align*}
and $R^{++}(\gamma):=Q(x,L)\times[t+(1+2\gamma)L^p,t+(2+\gamma)L^p)$,
where $\gamma$ is called the \emph{time lag}. Denote by $\mathcal{R}$ the set
of all parabolic rectangles in $\mathbb{R}^{n+1}$. Furthermore, the time lag
naturally appears in the parabolic Harnack estimate corresponding to
\eqref{20260811.1443}, which asserts that, for any given time lag
$\gamma\in(0,1)$, there exists a positive constant $C$ such that, for any
positive weak solution $u$ to \eqref{20260811.1443} and for any
$R\in\mathcal{R}$,
\begin{align}\label{20260811.1541}
\sup_{(x,t)\in R^-(\gamma)}u(x,t)
\leq C\inf_{(y,s)\in R^+(\gamma)}u(y,s).
\end{align}
Indeed, let $u$ be a positive weak solution to \eqref{20260811.1443}. Via a
parabolic John--Nirenberg inequality for $\ln u$, one can show that $\ln u$
belongs to the parabolic BMO space with time lag $\gamma$. Combining this and
the Moser iteration argument, one then obtains \eqref{20260811.1541} (see, for
example, \cite{a(tams-1988), kk(ma-2007)}).

Beyond its PDE origins, the parabolic BMO space with time lag also serves as
the higher-dimensional parabolic counterpart of the one-dimensional one-sided
BMO space. Recall that, on the real line $\mathbb{R}$, the \emph{one-sided
$\mathrm{BMO}$ space $\mathrm{BMO}^+(\mathbb{R})$}, introduced by
Mart\'in-Reyes and de la Torre \cite{mt(jlms-1994)}, is defined to be the set
of all locally integrable functions $f$ on $\mathbb{R}$ such that
\begin{align*}
\|f\|_{\mathrm{BMO}^+(\mathbb{R})}:=\sup_{I}\frac{1}{|I^-|}
\int_{I^-}\left[f(x)-\frac{1}{I^+}\int_{|I^+|}f(y)\,dy\right]_+\,dx<\infty,
\end{align*}
where the supremum is taken over all finite intervals
$I:=[a,b]\subset\mathbb{R}$,
$I^-:=[a,\frac{a+b}{2}]$, and $I^+:=[\frac{a+b}{2},b]$. Using the one-sided
John--Nirenberg inequality for functions in $\mathrm{BMO}^+(\mathbb{R})$, they
proved in \cite[Theorem 2]{mt(jlms-1994)} that $f\in\mathrm{BMO}^+(\mathbb{R})$
if and only if
\begin{align*}
|||f|||_{\mathrm{BMO}^+(\mathbb{R})}
:=\sup_{I}\inf_{c\in\mathbb{R}}\left\{\frac{1}{|I^-|}\int_{I^-}[f(x)-c]_+\,dx
+\frac{1}{|I^+|}\int_{I^+}[f(x)-c]_-\,dx\right\}<\infty,
\end{align*}
where the supremum is taken over all finite intervals $I\subset\mathbb{R}$.
Moreover, the two $\mathrm{BMO}^+(\mathbb{R})$-norms
$\|\cdot\|_{\mathrm{BMO}^+(\mathbb{R})}$ and
$|||\cdot|||_{\mathrm{BMO}^+(\mathbb{R})}$ are equivalent.
From now on, we work in $\mathbb{R}^{n+1}$. In the parabolic
setting, for any given time lag $\gamma\in[0,1)$, the \emph{parabolic
$\mathrm{BMO}$ space $\mathrm{PBMO}^-(\gamma)$ with time lag $\gamma$} is
defined to be the set of all locally integrable functions $f$ on
$\mathbb{R}^{n+1}$ such that
\begin{align*}
\|f\|_{\mathrm{PBMO}^-(\gamma)}:=\sup_{R\in\mathcal{R}}\inf_{c\in\mathbb{R}}
\left[\fint_{R^-(\gamma)}(f-c)_++\fint_{R^+(\gamma)}(f-c)_-\right]<\infty.
\end{align*}
If the condition above holds with the time axis reversed, then
$f\in\mathrm{PBMO}^+(\gamma)$ (the other \emph{parabolic $\mathrm{BMO}$ space
with time lag $\gamma$}). Here, and thereafter, when no confusion can arise,
we omit the differential $dx\,dt$ from integrals. Denote
by $L_\mathrm{loc}^1$ the set of all locally integrable functions on
$\mathbb{R}^{n+1}$ and, for any $f\in L_\mathrm{loc}^1$ and any measurable subset
$A\subset\mathbb{R}^{n+1}$ with $|A|\in(0,\infty)$, let
\begin{align*}
f_A:=\fint_Af:=\frac{1}{|A|}\int_Af,
\end{align*}
$f_+:=\max\{f,\,0\}$, and $f_-:=-\min\{f,\,0\}$. It is worth pointing out that,
from the perspective of function spaces, $\mathrm{PBMO}^-(\gamma)$ is clearly a
natural higher-dimensional parabolic counterpart of the space
$\mathrm{BMO}^+(\mathbb{R})$ equipped with the equivalent norm.

Inspired by the equivalence between $\|\cdot\|_{\mathrm{BMO}^+(\mathbb{R})}$
and $|||\cdot|||_{\mathrm{BMO}^+(\mathbb{R})}$ in the one-dimensional setting,
as well as the intrinsic connections between parabolic Muckenhoupt weights and
parabolic BMO spaces with time lag, Kinnunen and Saari \cite{ks(na-2016)}
introduced the \emph{one-sided parabolic $\mathrm{BMO}$ spaces
$\mathrm{BMO}^+(\gamma)$} and \emph{$\mathrm{BMO}^-(\gamma)$ with time lag
$\gamma\in[0,1)$}, which are defined, respectively, to be the set of all $f\in
L_\mathrm{loc}^1$ such that
\begin{align*}
\|f\|_{\mathrm{BMO}^+(\gamma)}:=\sup_{R\in\mathcal{R}}\fint_{R^-(\gamma)}
\left[f-f_{R^+(\gamma)}\right]_+<\infty
\end{align*}
and
\begin{align*}
\|f\|_{\mathrm{BMO}^-(\gamma)}:=\sup_{R\in\mathcal{R}}\fint_{R^+(\gamma)}
\left[f-f_{R^-(\gamma)}\right]_+<\infty.
\end{align*}
While it was shown in \cite[Proposition 4.4]{ks(na-2016)} that
\begin{align*}
\mathrm{PBMO}^-(\gamma)=\mathrm{BMO}^+(\gamma)
\cap\left[-\mathrm{BMO}^-(\gamma)\right],
\end{align*}
the fundamental question regarding the complete equivalence of these spaces
remained open. Specifically, Kinnunen and Saari
\cite[Question 4.5]{ks(na-2016)} posed the following question:
\begin{question}\label{question A}
Let $\gamma\in(0,1)$. Is it true that $\mathrm{BMO}^+(\gamma)
=-\mathrm{BMO}^-(\gamma)=\mathrm{PBMO}^-(\gamma)$?
\end{question}

Parallel to parabolic BMO spaces, the parabolic Muckenhoupt weights with time
lag introduced by Kinnunen and Saari \cite{ks(na-2016), ks(apde-2016)} and
Kinnunen and Myyryl\"ainen \cite{km(am-2024)} have also been extensively
studied. In what follows, by a \emph{weight} $w$, we mean a nonnegative locally integrable
function on $\mathbb{R}^{n+1}$ which is positive almost everywhere. Let
$r\in[1,\infty]$ and $\gamma\in[0,1)$. The \emph{parabolic Muckenhoupt class
$A_r^+(\gamma)$ with time lag $\gamma$} is defined to be the set of all
weights $w$ on $\mathbb{R}^{n+1}$ such that
\begin{align*}
[w]_{A_r^+(\gamma)}:=
\begin{cases}
\displaystyle
\sup_{R\in\mathcal{R}}\fint_{R^-(\gamma)}w
\left[\mathop\mathrm{ess\,inf}_{R^+(\gamma)}w\right]^{-1}<\infty
&\mathrm{if}\ r=1,\\
\displaystyle
\sup_{R\in\mathcal{R}}\fint_{R^-(\gamma)}w
\left[\fint_{R^+(\gamma)}w^\frac{1}{1-r}\right]^{r-1}<\infty
&\mathrm{if}\ r\in(1,\infty),\\
\displaystyle
\sup_{R\in\mathcal{R}}\fint_{R^-(\gamma)}w
\exp\left\{\fint_{R^+(\gamma)}\ln\frac1w\right\}<\infty
&\mathrm{if}\ r=\infty.\\
\end{cases}
\end{align*}
If the condition above holds with the time axis reversed, then $w\in
A_r^-(\gamma)$ (the other \emph{parabolic Muckenhoupt class with time lag}).
It is well known that, in the classical Muckenhoupt weight theory, as well as
in the one-dimensional one-sided weight theory of Mart\'in-Reyes et al.
\cite{mpt(cjm-1993)}, the $A_\infty$ class coincides with the union of
all $A_r$ classes with $r\in[1,\infty)$. However, in the parabolic setting,
whether $A_\infty^+(\gamma)$ coincides with the union of all $A_r^+(\gamma)$
classes with $r\in[1,\infty)$ remains a fundamental open problem. Specifically,
Kinnunen and Saari \cite[Question 4.5]{ks(na-2016)}
also raised the following question:
\begin{question}\label{question B}
Let $\gamma\in(0,1)$. Is it true that $A_\infty^+(\gamma)
=\bigcup_{r\in[1,\infty)}A_r^+(\gamma)$?
\end{question}
Although this identity remains open, the structure and some interesting
characterizations of $\bigcup_{r\in[1,\infty)}A_r^+(\gamma)$ have recently been
extensively investigated by Kinnunen and Myyry\"ainen \cite{km(am-2024)}. We
refer to \cite{kmy(ma-2023), kyy(2602.09741), kyyz(cvpde-2025), mos(2509.24486),
my(mz-2024),s(rmi-2016), s(ampa-2018)} for more recent studies on
$\mathrm{PBMO}^-(\gamma)$ and to \cite{ckyyz(jga-2026), cm(rmc-2026),
km(jam-2025), kmyz(pa-2023), kyyz(cjm-2025), lmv(2604.12561), mhy(fm-2023)} for
more recent studies on the parabolic Muckenhoupt class with time lag.

A third fundamental notion is the parabolic reverse H\"older class introduced
by Kinnunen and Saari \cite{ks(na-2016), ks(apde-2016)} and
Kinnunen and Myyryl\"ainen \cite{km(jam-2025)}. Let $q\in(1,\infty]$ and
$\gamma\in[0,1)$. The \emph{parabolic reverse H\"older class $RH_q^+(\gamma)$
with time lag} is defined to be the set of all weights $w$ on $\mathbb{R}^{n+1}$ such that
\begin{align*}
[w]_{RH_q^+(\gamma)}:=
\begin{cases}
\displaystyle
\sup_{R\in\mathcal{R}}\left[\fint_{R^-(\gamma)}w^q\right]^\frac1q
\left[\fint_{R^+(\gamma)}w\right]^{-1}<\infty
&\mathrm{if}\ q\in(1,\infty),\\
\displaystyle
\sup_{R\in\mathcal{R}}\left[\mathop\mathrm{ess\,sup}_{R^-(\gamma)}w\right]
\left[\fint_{R^+(\gamma)}w\right]^{-1}<\infty
&\mathrm{if}\ q=\infty.
\end{cases}
\end{align*}
If the condition above holds with the time axis reversed, then $w\in
RH_q^-(\gamma)$ (the other \emph{parabolic reverse H\"older class with time
lag}). For simplicity, we abbreviate $RH_q^\pm(0)$ as $RH_q^\pm$. It is well
known that, in both the classical Euclidean setting and the one-dimensional
one-sided theory of Cruz-Uribe et al. \cite{cno(sm-1995)}, the reverse Hölder
condition is equivalent to the Muckenhoupt condition. In the parabolic setting,
it was proved by Kinnunen and Saari \cite[Theorem 5.2]{ks(apde-2016)} and
Kinnunen and Myyryl\"ainen \cite[Theorem 5.2]{km(am-2024)} that, for any given
$\gamma\in(0,1)$, the inclusion $\bigcup_{r\in[1,\infty)}A_r^+(\gamma)\subset
\bigcup_{q\in(1,\infty]}RH_q^+$ holds. However, whether the reverse inclusion
holds remained an open problem. Specifically, Kinnunen and Saari \cite[Question
4.6]{ks(na-2016)} posed the following question:
\begin{question}\label{question C}
Let $\gamma\in(0,1)$. Is it true that
$\bigcup_{q\in(1,\infty]}RH_q^+=\bigcup_{q\in[1,\infty)}A_q^+(\gamma)$?
\end{question}
Recently, partial progress toward Question \ref{question C} has been made by
Kinnunen and Myyryl\"ainen \cite[Theorem 5.3]{km(jam-2025)}. They showed that,
if $w\in RH_q^+$ for some $q\in(1,\infty)$ satisfies the following
additional \emph{parabolic forward in time doubling condition}: given
$\gamma\in(0,1)$, there exists a positive constant
$C_d:=C_d(n,p,q,\gamma,[w]_{RH_q^+})$ such that, for any
$R=(x,t,L)\in\mathcal{R}$ with $(x,t)\in\mathbb{R}^{n+1}$ and $L\in(0,\infty)$,
\begin{align}\label{20260730.1813}
w\left(R^-(\gamma)\right)\leq C_dw\left(\frac12R^+(\gamma)\right),
\end{align}
then $w\in A_r^+(\gamma)$ for some $r\in[1,\infty)$. Here, and thereafter,
\begin{align*}
\frac12R^+(\gamma):=Q\left(x,\frac{L}{2}\right)\times
\left[t+\frac{1+\gamma}{2}L^p-\frac{1-\gamma}{2}\frac{L^p}{2^p},
t+\frac{1+\gamma}{2}L^p+\frac{1-\gamma}{2}\frac{L^p}{2^p}\right).
\end{align*}

In this article, we give affirmative answers to Questions
\ref{question A}--\ref{question C}. Specifically,
we first asnwer Question \ref{question C} as follows.

\begin{theorem}\label{question 4.6}
Let $\gamma\in(0,1)$. Then
\begin{align*}
\bigcup_{q\in(1,\infty]}RH_q^+=\bigcup_{r\in[1,\infty)}A_r^+(\gamma).
\end{align*}
More precisely, if $w\in A_r^+(\gamma)$ for some $r\in[1,\infty)$, then there
exists a positive constant $q:=q(n,p,r,\gamma,[w]_{A_r^+(\gamma)})
\in(1,\infty]$ such that $w\in RH_q^+$. Conversely, if $w\in RH_q^+$ for some
$q\in(1,\infty]$, then there exists a positive constant
$r:=r(n,p,q,\gamma,[w]_{RH_q^+})\in[1,\infty)$ such that $w\in A_r^+(\gamma)$.
\end{theorem}

To overcome the essential difficulty arising from the time lag and the
parabolic geometry, we first establish the uniform parabolic space-time shifting
property for weights in $\bigcup_{q\in(1,\infty]}RH_q^+$ (see Theorem
\ref{shifting}), which allows a parabolic reverse H\"older weight to be moved
not only forward in time but also in spatial directions, at the cost of a
controlled shift. After that, we use this uniform parabolic space-time shifting
property to prove that, if $w\in\bigcup_{q\in(1,\infty]}RH_q^+$, then $w$
automatically satisfies the parabolic forward in time doubling condition
\eqref{20260730.1813}, thereby completing the proof of Theorem
\ref{question 4.6}.

Next, using a new one-sided stopping time argument, we establish a parabolic
John--Nirenberg inequality for functions in $\mathrm{BMO}^+(\gamma)$.
Then, applying this John--Nirenberg inequality and Theorem \ref{question 4.6},
we answer Questions \ref{question A} and \ref{question B} as follows.

\begin{theorem}\label{question 4.5}
Let $\gamma\in(0,1)$. Then the following two statements hold.
\begin{enumerate}
\item[\rm(i)] $\mathrm{PBMO}^-(\gamma)=\mathrm{BMO}^+(\gamma)=
-\mathrm{BMO}^-(\gamma)$. Moreover, for any $f\in L_\mathrm{loc}^1$,
\begin{align*}
\|f\|_{\mathrm{PBMO}^-(\gamma)}\sim\|f\|_{\mathrm{BMO}^+(\gamma)}\sim
\|-f\|_{\mathrm{BMO}^-(\gamma)},
\end{align*}
where the positive equivalence constants are independent of $f$.

\item[\rm(ii)] $A_\infty^+(\gamma)=\bigcup_{r\in[1,\infty)}A_r^+(\gamma)$.
\end{enumerate}
\end{theorem}

We point out that, on the real line, any open set can be uniquely represented
as a union of at most countably many disjoint open connected intervals.
However, this fundamental topological property fails in higher dimensions.
Furthermore, in the parabolic geometry, the temporal variable scales as the
$p$-th power of the spatial variables. Consequently, when applying a dyadic
decomposition to a parabolic rectangle, one would like to divide the temporal
edge into $2^p$ equally long intervals. However, $p$ is not necessarily an
integer, which causes the shape of the resulting subrectangles to be different
from the parent rectangle. In addition, the presence of the time lag introduces
a strictly positive temporal gap between the relevant rectangles $R^-(\gamma)$
and $R^+(\gamma)$. These intrinsic geometric obstacles render both the
one-dimensional one-sided techniques and the classical Euclidean dyadic tools
inapplicable in the present setting. To overcome these difficulties, we must
fully exploit the specific parabolic rectangle dyadic decomposition and develop
new techniques, including a uniform parabolic space-time shifting strategy and
a new one-sided stopping time argument.

The organization of the remainder of this article is as follows.
In Section \ref{section2}, we first show that, if $w\in RH_q^+$ for some
$q\in(1,\infty]$, then $[w]_{RH_q^+(\gamma)}$ is uniformly bounded for $\gamma$
in any compact subinterval of $[0,1)$ (see Lemma \ref{Kim lemma 1}). Combining this and
a parabolic finite overlapping lemma (see Lemma \ref{finite overlapping}),
we establish the uniform parabolic space-time shifting property for weights in
$\bigcup_{q\in(1,\infty]}RH_q^+$ (see Theorem \ref{shifting}). Finally, we use
Theorem \ref{shifting} to prove Theorem \ref{question 4.6}. In Section
\ref{section3}, we first establish a parabolic John--Nirenberg inequality
for functions in $\mathrm{BMO}^+(\gamma)$ (see Lemma \ref{J-N for BMO+}), which
can be viewed as a higher-dimensional counterpart to the one-sided
John–Nirenberg inequality for functions in $\mathrm{BMO}^+(\mathbb{R})$
established in \cite[Theorem 3]{mt(jlms-1994)}. Then we use Lemma
\ref{J-N for BMO+} and Theorem \ref{question 4.6} to show Theorem
\ref{question 4.5}. As applications, we characterize $\mathrm{BMO}^+(\gamma)$
in terms of a new parabolic John--Nirenberg inequality and exponential
integrability (see Corollary \ref{J-N for BMO+ 2}); moreover, further
applications also give that $\mathrm{BMO}^+(\gamma)$ is independent of the time
lag (see Corollary \ref{BMO+ time lag}), the null space of
$\mathrm{BMO}^+(\gamma)$ consists of all non-decreasing functions which depend
only on the time variable (see Corollary \ref{BMO+ null space}), and the
logarithm of any positive weak solution to the equation \eqref{20260811.1443}
belongs to $\mathrm{BMO}^+(\gamma)$ (see Corollary \ref{solutions J-N}).

We end this introduction by making some notational conventions. Throughout
this article, let $\mathbb{N}:=\{1,\,2,\,\dots\}$ and
$\mathbb{Z}_+:=\mathbb{N}\cup\{0\}$. Let $\mathbf{0}$ denote the origin of
$\mathbb{R}^n$. For any $s\in\mathbb{R}$, the \emph{notation}
$\lceil s\rceil$ denotes the smallest integer not less than $s$ and the
\emph{notation} $\lfloor s\rfloor$ denotes the largest integer not greater than
$s$. For any given $q\in[1,\infty]$, we denote by $q'$ its \emph{conjugate
exponent}, i.e., $\frac1q+\frac{1}{q'}=1$. For any $x:=(x_1,\dots,x_n),
y:=(y_1,\dots,y_n)\in\mathbb{R}^n$, let
$$\|x-y\|_\infty:=\max\{|x_1-y_1|,\dots,|x_n-y_n|\}.$$
The \emph{notation} $f\lesssim g$ means that there exists a positive constant $C$
such that $f\leq Cg$. If $f\lesssim g$ and $g\lesssim f$, then we write $f\sim
g$. If $f\leq Cg$ and $g=h$ or $g\leq h$, we then write $f\lesssim g=h$ or
$f\lesssim g\leq h$. For any measurable set $A\subset\mathbb{R}^{n+1}$, we
denote by $|A|$ its Lebesgue measure. For any rectangle
$R:=Q(x,L)\times[t,T)\subset\mathbb{R}^{n+1}$ with
$x\in\mathbb{R}^n$, $L\in(0,\infty)$, and $-\infty<t<T<\infty$, let
$l_x(R):=L$ and $l_t(R):=T-t$. For any $A\subset\mathbb{R}^{n+1}$, denote by
the \emph{notation} $\mathrm{pr}_x(A)$ the projection of $A$ to $\mathbb{R}^n$
and by the \emph{notation} $\mathrm{pr}_t(A)$ the projection of $A$ to the time
axis. Finally, in all proofs, we consistently retain the notation introduced in
the original theorem (or related statement).

\section{Proof of Theorem \ref{question 4.6}}
\label{section2}

In this section, we prove Theorem \ref{question 4.6}. To begin with, we recall
some existing developments regarding this problem. The following
Lemma \ref{Kim lem 2}(i) is precisely \cite[Theorem 5.2]{km(am-2024)}
and Lemma \ref{Kim lem 2}(ii) is exactly \cite[Theorem 5.3]{km(jam-2025)}.

\begin{lemma}\label{Kim lem 2}
Let $q\in(1,\infty]$, $r\in[1,\infty)$, and $\gamma\in(0,1)$. Then the following
assertions hold.
\begin{enumerate}
\item[\rm(i)] If $w\in A_r^+(\gamma)$, then there exists
$q:=q(n,p,r,\gamma,[w]_{A_r^+(\gamma)})\in(1,\infty)$ such that $w\in RH_q^+$.

\item[\rm(ii)] Conversely, if $w\in RH_q^+$ satisfies \eqref{20260730.1813},
then there exists $r:=r(n,p,q,\gamma,C_d,[w]_{RH_q^+})\in(1,\infty)$ such that
$w\in A_r^+(\gamma)$.
\end{enumerate}
\end{lemma}

By Lemma \ref{Kim lem 2}, we find that, to show Theorem \ref{question 4.6},
it suffices to prove that the $RH_q^+$ condition automatically implies the
parabolic doubling condition \eqref{20260730.1813}. To this end, we first
show that $RH_q^+(\gamma)$ does not depend on the time lag $\gamma$.
Moreover, $[w]_{RH_q^+(\gamma)}$ is uniformly bounded for $\gamma$ in any
compact subinterval of $[0,1)$.

\begin{lemma}\label{Kim lemma 1}
Let $q\in(1,\infty]$, $\gamma\in(0,1)$, and $w$ be a weight. Then $w\in
RH_q^+$ if and only if $w\in RH_q^+(\gamma)$. Furthermore, if $\gamma^*\in
(0,1)$ and $w\in RH_q^+$, then there exists a positive constant
$C(n,p,q,\gamma^*,[w]_{RH_q^+})$ such that
\begin{align}\label{20260730.2216}
\sup_{\gamma\in[0,\gamma^*]}[w]_{RH_q^+(\gamma)}\leq
C\left(n,p,q,\gamma^*,[w]_{RH_q^+}\right).
\end{align}
\end{lemma}

\begin{proof}
From \cite[Lemma 2.5]{km(jam-2025)}, we deduce that the first statement holds,
i.e., $w\in RH_q^+$ if and only if $w\in RH_q^+(\gamma)$. We turn to prove
the second assertion. Let $\gamma^*\in(0,1)$ and $w\in RH_q^+$. We first assume
that $q<\infty$ and we show that there exists a positive constant
$C_1(n,p,q,\gamma^*,[w]_{RH_q^+})$ such that, for any $R\in\mathcal{R}$ and
$\lambda\in[\frac12,1]$,
\begin{align}\label{20260730.2218}
\left[\fint_{R^-(\gamma)}w^q\right]^\frac1q\leq
C_1\left(n,p,q,\gamma^*,[w]_{RH_q^+}\right)\fint_{R^-(\gamma)+(\mathbf{0},
\lambda(1-\gamma)L^p)}w.
\end{align}
Indeed, let $R\in\mathcal{R}$ and $\lambda\in[\frac12,1]$. Define
$l:=\lambda^\frac1p(1-\gamma)^\frac1pL$. Then
$l\in[(\frac{1-\gamma^*}{2})^\frac1pL,L]$. We cover each spatial edge of
$R^-(\gamma)$ by $\lceil\frac{L}{l}\rceil$ half-open intervals of equal length
$l$, with an overlap bounded by 2. We also cover the temporal edge of
$R^-(\gamma)$ by 2 equally long half-open intervals with length
$l^p$. Then we obtain $(\lceil\frac{L}{l}\rceil)^n$ spatial
cubes $\{Q_i\}_{i\in\mathbb{N}\cap[1,(\lceil\frac{L}{l}\rceil)^n]}$ with edge
length $l$, and two temporal intervals $J_1$ and $J_2$ with edge length $l^p$.
For any $i\in\mathbb{N}\cap[1,(\lceil\frac{L}{l}\rceil)^n]$ and $j\in\{1,\,2\}$,
define $S_{i,j}^-:=Q_i\times J_j$ and $S_{i,j}^+:=S_{i,j}^-+(\mathbf{0},l^p)$.
Then
\begin{align*}
R^-(\gamma)=\bigcup_{i\in\mathbb{N}\cap[1,(\lceil\frac{L}{l}\rceil)^n],
\,j\in\{1,\,2\}}S_{i,j}^-
\end{align*}
and
\begin{align*}
R^-(\gamma)+\left(\mathbf{0},\lambda(1-\gamma)L^p\right)
=\bigcup_{i\in\mathbb{N}\cap[1,(\lceil\frac{L}{l}\rceil)^n],
\,j\in\{1,\,2\}}S_{i,j}^+
\end{align*}
both with an overlap bounded by $2^{n+1}$. Combining this and the proven
conclusion that $l\in[(\frac{1-\gamma^*}{2})^\frac1pL,L]$ further implies that
\begin{align*}
\left[\fint_{R^-(\gamma)}w^q\right]^\frac1q
&\leq\left[\sum_{i\in\mathbb{N}\cap[1,(\lceil\frac{L}{l}\rceil)^n],
\,j\in\{1,\,2\}}\frac{|S_{i,j}^-|}{|R^-(\gamma)|}
\fint_{S_{i,j}^-}w^q\right]^\frac1q\\
&\leq\sum_{i\in\mathbb{N}\cap[1,(\lceil\frac{L}{l}\rceil)^n],\,j\in\{1,\,2\}}
\left[\frac{|S_{i,j}^-|}{|R^-(\gamma)|}\fint_{S_{i,j}^-}w^q\right]^\frac1q\\
&\leq[w]_{RH_q^+}\sum_{i\in\mathbb{N}\cap[1,(\lceil\frac{L}{l}\rceil)^n],
\,j\in\{1,\,2\}}\left[\frac{|S_{i,j}^-|}{|R^-(\gamma)|}\right]^\frac1q
\fint_{S_{i,j}^+}w\\
&\leq\left(\frac{2}{1-\gamma^*}\right)^\frac{n+p}{pq'}2^{n+1}[w]_{RH_q^+}
\fint_{R^-(\gamma)+(\mathbf{0},\lambda(1-\gamma)L^p)}w.
\end{align*}
Thus, \eqref{20260730.2218} holds with $C_1(n,p,q,\gamma^*,[w]_{RH_q^+}):=
(\frac{2}{1-\gamma^*})^\frac{n+p}{pq'}2^{n+1}[w]_{RH_q^+}$.

Now, we show \eqref{20260730.2216}. Let $R\in\mathcal{R}$,
$N:=\lceil\frac{1+\gamma}{1-\gamma}\rceil$, and
$\lambda\in\frac{\frac{1+\gamma}{1-\gamma}}{N}$. Then
$N\leq\lceil\frac{1+\gamma^*}{1-\gamma^*}\rceil$ and $\lambda\in[\frac12,1]$.
For any $k\in[0,N]$, define $P_k:=R^-(\gamma)+
(\mathbf{0},k\lambda(1-\gamma)L^p)$. Then $P_0=R^-(\gamma)$ and
$P_N=R^+(\gamma)$. Moreover, applying \eqref{20260730.2218} and H\"older's
inequality, we find that
\begin{align*}
\left[\fint_{R^-(\gamma)}w^q\right]^\frac1q&\leq
C_1\left(n,p,q,\gamma^*,[w]_{RH_q^+}\right)\fint_{P_1}w\\
&\leq C_1\left(n,p,q,\gamma^*,[w]_{RH_q^+}\right)
\left[\fint_{P_1}w^q\right]^\frac1q\\
&\leq C_1\left(n,p,q,\gamma^*,[w]_{RH_q^+}\right)^N\fint_{P_N}w\\
&\leq C_1\left(n,p,q,\gamma^*,[w]_{RH_q^+}\right)
^{\lceil\frac{1+\gamma^*}{1-\gamma^*}\rceil}\fint_{R^+(\gamma)}w.
\end{align*}
Taking the supremum over $R\in\mathcal{R}$, we conclude that
\eqref{20260730.2216} holds with
$$C(n,p,q,\gamma^*,[w]_{RH_q^+}):=
C_1(n,p,q,\gamma^*,[w]_{RH_q^+})^{\lceil\frac{1+\gamma^*}{1-\gamma^*}\rceil}.$$
By letting $q\to\infty$, we obtain the same conclusion for $RH_\infty^+$,
which then completes the proof of Lemma \ref{Kim lemma 1}.
\end{proof}

We need to use the following basic properties of $RH_q^+$, which are inherently
inspired by \cite[Theorem 3.2 and Lemma 3.3]{km(jam-2025)}. It is worth
emphasizing that, unlike those in \cite{km(jam-2025)}, our constants in
Lemma \ref{Kim lemma 1 cor} depend only on the upper bound $\gamma^*$ of the
time lag $\gamma\in[0,\gamma^*]$, thanks to Lemma \ref{Kim lemma 1}. In what
follows, for any $A\subset\mathbb{R}^{n+1}$ and $(x,t)\in\mathbb{R}^{n+1}$, define
\begin{align*}
A+(x,t):=\{(y+x,s+t):(y,s)\in A\}.
\end{align*}

\begin{lemma}\label{Kim lemma 1 cor}
Let $q\in(1,\infty]$, $\gamma^*\in(0,1)$, $\gamma\in[0,\gamma^*]$,
$\Theta\in(0,\infty)$, and $w\in RH_q^+$. Then the following statements hold.
\begin{enumerate}
\item[\rm(i)] For any $R\in\mathcal{R}$ and any measurable subset
$E\subset R^-(\gamma)$,
\begin{align*}
\frac{w(E)}{w(R^+(\gamma))}\leq C\left(n,p,q,\gamma^*,[w]_{RH_q^+}\right)
\left[\frac{|E|}{|R^-(\gamma)|}\right]^\frac{1}{q'},
\end{align*}
where $C(n,p,q,\gamma^*,[w]_{RH_q^+})$ is the same as in \eqref{20260730.2216}.

\item[\rm(ii)] There exists a positive constant $C(n,p,q,\gamma^*,\Theta,
[w]_{RH_q^+})$ such that, for any $R\in\mathcal{R}$ and $\theta\in[0,\Theta]$,
\begin{align*}
w\left(R^-(\gamma)\right)\leq C\left(n,p,q,\gamma^*,\Theta,
[w]_{RH_q^+}\right)w\left(R^-(\gamma)
+\left(\mathbf{0},\theta[l(R)]^p\right)\right).
\end{align*}
\end{enumerate}
\end{lemma}

\begin{proof}
We first prove (i). Let $R\in\mathcal{R}$ and $E\subset R^-(\gamma)$ be
measurable. From H\"older's inequality and Lemma \ref{Kim lemma 1}, we infer that
\begin{align*}
\frac{w(E)}{w(R^+(\gamma))}&=\frac{|E|}{w(R^+(\gamma))}\fint_Ew
\leq\frac{|E|}{w(R^+(\gamma))}\left(\fint_Ew^q\right)^\frac1q\\
&\leq\frac{|E|^\frac{1}{q'}}{|R^-(\gamma)|^\frac1qw(R^-(\gamma))}
\left[\fint_{R^-(\gamma)}w^q\right]^\frac1q
\leq\frac{[w]_{RH_q^+(\gamma)}|E|^\frac{1}{q'}}
{|R^-(\gamma)|^\frac1qw(R^-(\gamma))}\fint_{R^+(\gamma)}w\\
&\leq C\left(n,p,q,\gamma^*,[w]_{RH_q^+}\right)
\left[\frac{|E|}{|R^-(\gamma)|}\right]^\frac{1}{q'},
\end{align*}
which then completes the proof of (i). By (i) and the proof of \cite[Lemma
3.3(ii)]{km(jam-2025)} with
$\alpha:=\frac{1}{[2\max\{1,\,C(n,p,q,\gamma^*,[w]_{RH_q^+})1\}]^{q'}}$ and
$\beta:=\frac12$ therein, we find that (ii) holds with
\begin{align*}
C\left(n,p,q,\gamma^*,\Theta,[w]_{RH_q^+}\right)=
4\left[2\max\left\{1,\,C\left(n,p,q,\gamma^*,[w]_{RH_q^+}\right)\right\}\right]
^{q'(1+\frac{2^{p+1}\Theta}{1-\gamma^*})}.
\end{align*}
This completes the proof of Lemma \ref{Kim lemma 1 cor}.
\end{proof}

Now, we recall the parabolic dyadic lattice introduced in
\cite{lmv(2604.12561)}. Let $\gamma\in[0,\frac12]$ and $R\in\mathcal{R}$. The
\emph{parabolic dyadic lattice $\mathcal{D}_1(R^-(\gamma))$} of $R^-(\gamma)$
is defined as follows. Divide each spatial edge of $R^-(\gamma)$ into $2^4$
equally long half-open intervals and partition the temporal edge of
$R^-(\gamma)$ into $k$ equally long half-open intervals, where
\begin{align*}
k:=\begin{cases}
\displaystyle
\lceil2^{4p}\rceil &\displaystyle\mbox{if\ }0\leq\gamma\leq1-\frac{2^{4p}+1}{2^{4p+1}},\\
\displaystyle
\lfloor2^{4p}\rfloor &\displaystyle\mbox{if\ }
1-\frac{2^{4p}+1}{2^{4p+1}}\leq\gamma\leq\frac12.
\end{cases}
\end{align*}
Then we obtain $J$ half-open subrectangles $\{P_i\}_{i\in\mathbb{N}\cap[1,J]}$
of $R^-(\gamma)$, where $J$ equals $2^{4n}\lceil2^{4p}\rceil$ or
$2^{4n}\lfloor2^{4p}\rfloor$. Define
$\mathcal{D}_1(R^-(\gamma)):=\{P_i\}_{i\in\mathbb{N}\cap[1,J]}$.
From \cite[Proposition 2.1]{lmv(2604.12561)}, we deduce that the following
assertions hold.
\begin{Romanlist}
\item\label{(I)} $\{P_i\}_{i\in\mathbb{N}\cap[1,J]}$ are pairwise disjoint and
$R^-(\gamma)=\bigcup_{i\in\mathbb{N}\cap[1,J]}P_i$.

\item\label{(II)} For any $i\in\mathbb{N}$, $l_x(P_i)=\frac{l(R)}{2^4}$ and
$\frac12\frac{(1-\gamma)[l(R)]^p}{2^{4p}}\leq
l_t(P_i)\leq2\frac{(1-\gamma)[l(R)]^p}{2^{4p}}$. Moreover, there exists
$\alpha\in[0,\frac12]$, independent of $R$, such that, for any
$i\in\mathbb{N}\cap[1,J]$, there exists $R_i\in\mathcal{R}$ satisfying
$P_i=R_i^-(\alpha)$.
\end{Romanlist}

Based on the above parabolic dyadic lattice, we then show a parabolic finite
overlapping lemma. Let $a,b\in(0,\infty)$, $Q\subset\mathbb{R}^n$ be a cube,
$I\subset\mathbb{R}$ be an interval, and $R:=Q\times I$. Define the
\emph{anisotropic parabolic enlargement $\Omega_{a,b}(R)$} of $R$ by setting
\begin{align*}
\Omega_{a,b}(R):=\left\{(x,t)\in\mathbb{R}^{n+1}:\mathrm{dist}_\infty
(x,Q)<al(Q)\mbox{\ \ and\ \ } \mathrm{dist}(t,I)<b|I|\right\},
\end{align*}
where, for any $E\subset\mathbb{R}^n$,
$\mathrm{dist}_\infty(x,E):=\inf\{\|x-y\|_\infty:y\in E\}$.

\begin{lemma}\label{finite overlapping}
Let $a,b\in(0,\infty)$ and $\gamma\in[0,\frac12]$. Then there exists a positive
constant $B(n,a,b)$ such that, for any $R\in\mathcal{R}$,
\begin{align*}
\sum_{P\in\mathcal{D}_1(R^-(\gamma))}\boldsymbol{1}_{\Omega_{a,b}(P)}
\leq B(n,a,b).
\end{align*}
\end{lemma}

\begin{proof}
We first make the following two observations. Let $v_0,v\in\mathbb{R}$ and
$l\in(0,\infty)$. For any $m\in\mathbb{Z}$, define
$I_m(v_0,l):=[v_0+ml,v_0+(m+1)l)$. If there exists $m\in\mathbb{Z}$ such that
$\mathrm{dist}(v,I_m(v_0,l))<al$, then
$\frac{v-v_0}{l}-1-a<m<\frac{v-v_0}{l}+a$. Thus,
\begin{align}\label{20260803.1743}
\#\{m\in\mathbb{Z}:\mathrm{dist}(v,I_m(v_0,l))<al\}\leq\lceil1+2a\rceil+1.
\end{align}
Analogously, for any $m\in\mathbb{Z}$, define
\begin{align*}
J_m(v_0,l):=\left[v_0+m(1-\alpha)l^p,v_0+(m+1)(1-\alpha)l^p\right),
\end{align*}
where $\alpha$ is the same as in \ref{(II)}. If there exists $m\in\mathbb{Z}$
such that $\mathrm{dist}(v,J_m(v_0,l))<bl^p$, then
\begin{align*}
\frac{v-v_0}{(1-\alpha)l^p}-1-\frac{b}{1-\alpha}<m
<\frac{v-v_0}{(1-\alpha)l^p}+\frac{b}{1-\alpha}.
\end{align*}
This, together with the
fact that $\alpha\in[0,\frac12]$, further implies that
\begin{align}\label{20260803.1746}
\#\{m\in\mathbb{Z}:\mathrm{dist}(v,J_m(v_0,l))<bl^p\}\leq\lceil1+4b\rceil+1.
\end{align}

Now, let $R\in\mathcal{R}$. Note that there exist
$\{v_i\}_{i\in\mathbb{N}\cap[1,n+1]}$ such that
$$R^-(\gamma)=[v_1,v_1+l(R))\times\cdots\times[v_n,v_n+l(R))\times
\left[v_{n+1},v_{n+1}+(1-\gamma)[l(R)]^p\right)$$
and, for any
$P\in\mathcal{D}_1(R^-(\gamma))$, there exist
$\{m_i\}_{i\in\mathbb{N}\cap[1,n+1]}$ such that
$$P=I_{m_1}\left(v_1,\frac{l(R)}{2^4}\right)\times\cdots\times
I_{m_n}\left(v_n,\frac{l(R)}{2^4}\right)\times J_{m_{n+1}}\left(v_{n+1},\frac{l(R)}{2^4}\right).$$
Combining this, \eqref{20260803.1743}, and \eqref{20260803.1746},
we conclude that
\begin{align*}
\sum_{P\in\mathcal{D}_1(R^-(\gamma))}\boldsymbol{1}_{\Omega_{a,b}(P)}\leq
\left(\lceil1+2a\rceil+1\right)^n\left(\lceil1+4b\rceil+1\right),
\end{align*}
thereby completing the proof of Lemma \ref{finite overlapping} with $B(n,a,b):=
(\lceil1+2a\rceil+1)^n(\lceil1+4b\rceil+1)$.
\end{proof}

Now, we establish the uniform parabolic space-time shifting property of
parabolic reverse H\"older weights, which enables simultaneous spatial and
forward in time translations of such weights, at the cost of a controlled
shift. This is a key ingredient in the proof of Theorem \ref{question 4.6}. In
what follows, for any $i\in\mathbb{N}\cap[1,n]$, let $e_i$ denote the unit
vector in the $i$-th coordinate direction.

\begin{theorem}\label{shifting}
Let $q\in(1,\infty]$ and $w\in RH_q^+$. Then there exist $\nu\in(0,\frac14)$
and $\Theta,C(n,p,q,[w]_{RH_q^+})\in(0,\infty)$, all depending only on $n$,
$p$, $q$, and $[w]_{RH_q^+}$, such that, for any $\gamma\in[0,\frac12]$,
$R\in\mathcal{R}$, $i\in\mathbb{N}\cap[1,n]$, and $\sigma\in\{-1,\,1\}$,
\begin{align}\label{20260804.1606}
w\left(R^-(\gamma)\right)\leq C\left(n,p,q,[w]_{RH_q^+}\right)
w\left(R^-(\gamma)+\left(\sigma\nu l(R)e_i,\Theta[l(R)]^p\right)\right).
\end{align}
\end{theorem}

\begin{proof}
Let $\gamma\in[0,\frac12]$, $R\in\mathcal{R}$, $i\in\mathbb{N}\cap[1,n]$, and
$\sigma\in\{-1,\,1\}$. Without loss of generality, we may assume that
$\sigma:=1$, since the proof in the case $\sigma=-1$ is entirely similar.
Let $\nu\in(0,\frac14)$ be determined later. Define
$S_\nu^+(\gamma):=R^+(\gamma)+(\nu l(R)e_i,0)$ and
$E_\nu:=R^+(\gamma)\setminus S^+(\gamma)$.
We first prove that
\begin{align}\label{20260803.2206}
w\left(R^-(\gamma)\right)&\leq C\left(n,p,q,\frac12,[w]_{RH_q^+}\right)
w\left(S_\nu^+(\gamma)\right)\notag\\
&\quad+\left[C\left(n,p,q,\frac12,[w]_{RH_q^+}\right)\right]
^2\nu^\frac{1}{q'}w\left(R^{++}(\gamma)\right),
\end{align}
where $C(n,p,q,\frac12,[w]_{RH_q^+})$ is the same as in \eqref{20260730.2216}
with $\gamma^*:=\frac12$ therein. Indeed, from Lemma \ref{Kim lemma 1 cor}(i)
with $\gamma^*:=\frac12$ and $E:=R^-(\gamma)$ therein, it follows that
\begin{align}\label{20260803.2212}
w\left(R^-(\gamma)\right)\leq C\left(n,p,q,\frac12,[w]_{RH_q^+}\right)
w\left(R^+(\gamma)\right).
\end{align}
Moreover, note that $E_\nu\subset R^+(\gamma)$ and $|E_\nu|=\nu|R^+(\gamma)|$.
These, together with Lemma \ref{Kim lemma 1 cor}(i) with $\gamma^*:=\frac12$
and $E:=E_\nu$ therein, further implies that
\begin{align*}
w(E_\nu)\leq C\left(n,p,q,\frac12,[w]_{RH_q^+}\right)\nu^\frac{1}{q'}
w\left(R^{++}(\gamma)\right).
\end{align*}
Combining this, \eqref{20260803.2212}, and the fact that $R^+(\gamma)\subset
S_\nu^+(\gamma)\cup E_\nu$, we conclude that
\begin{align*}
w\left(R^-(\gamma)\right)&\leq C\left(n,p,q,\frac12,[w]_{RH_q^+}\right)
w\left(R^+(\gamma)\right)\\
&\leq C\left(n,p,q,\frac12,[w]_{RH_q^+}\right)
\left[w\left(S^+(\gamma)\right)+w(E_\nu)\right]\\
&\leq C\left(n,p,q,\frac12,[w]_{RH_q^+}\right)w\left(S_\nu^+(\gamma)\right)\\
&\quad+\left[C\left(n,p,q,\frac12,[w]_{RH_q^+}\right)\right]^2
\nu^\frac{1}{q'}w\left(R^{++}(\gamma)\right),
\end{align*}
and hence \eqref{20260803.2206} holds.

Now, we apply the parabolic dyadic lattice argument to $R^{++}(\gamma)$ and we
obtain a sequence $\{P\}_{P\in\mathcal{D}_1(R^{++}(\gamma))}$. Let
$\Theta_0:=0$ and, for any $k\in\mathbb{Z}_+$, define $\Theta_{k+1}:=3+
\frac{\Theta_k}{2^{4(p-1)}}$. Then, for any $k\in\mathbb{N}$,
$\Theta_k=3\sum_{j=0}^{k-1}\frac{1}{2^{4(p-1)j}}$ and $\Theta_k\uparrow\Theta:=
\frac{3}{1-\frac{1}{2^{4(p-1)}}}$ as $k\to\infty$. Fix
$b_0\in(\frac{3+\frac{\Theta}{2^{4(p-1)}}}{1-\frac{1}{2^{4p}}},\infty)$.
For any $j\in\mathbb{Z}_+\cap[0,2^4]$, $k\in\mathbb{Z}_+$, and
$P\in\mathcal{D}_1(R^{++}(\gamma))$, define
\begin{align*}
P_{j,k}:=P+\left(j\nu\frac{l(R)}{2^4}e_i,
j\Theta_k\left[\frac{l(R)}{2^4}\right]^p\right).
\end{align*}
We show that, for any given $j\in\mathbb{Z}_+\cap[0,2^4-1]$ and
$k\in\mathbb{Z}_+$,
\begin{align}\label{20260804.1415}
\sum_{P\in\mathcal{D}_1(R^{++}(\gamma))}
w\left(\Omega_{1,b_0}\left(P_{j,k}\right)\right)\leq B(n,1,b_0)
w\left(\Omega_{1,b_0}\left(R^-(\gamma)\right)\right),
\end{align}
where $B(n,1,b_0)$ is the same as in Lemma \ref{finite overlapping}.
To this end, we first observe that, for any $j\in\mathbb{Z}_+\cap[0,2^4-1]$,
$k\in\mathbb{Z}_+$, and $P\in\mathcal{D}_1(R^{++}(\gamma))$, it holds that
$\Omega_{1,b_0}(P_{j,k})\subset\Omega_{1,b_0}(R^-(\gamma))$. Indeed, let
$(x,t)\in\Omega_{1,b_0}(P_{j,k})$. Since $\nu<\frac14$, it follows that
\begin{align*}
\mathrm{dist}_\infty\left(x,\mathrm{pr}_x\left(R^-(\gamma)\right)\right)
&\leq\mathrm{dist}_\infty\left(x,\mathrm{pr}_x\left(P_{j,k}\right)\right)+
\mathrm{dist}_\infty\left(\mathrm{pr}_x\left(P_{j,k}\right),
\mathrm{pr}_x\left(R^-(\gamma)\right)\right)\\
&<\frac{l(R)}{2^4}+\nu l(R)<l(R).
\end{align*}
Furthermore, by the assumptions that $\gamma\leq\frac12$ and
$b_0>\frac{3+\frac{\Theta}{2^{4(p-1)}}}{1-\frac{1}{2^{4p}}}$, we obtain
\begin{align*}
\mathrm{dist}\left(t,\mathrm{pr}_t\left(R^-(\gamma)\right)\right)
&\leq\mathrm{dist}\left(t,\mathrm{pr}_t\left(P_{j,k}\right)\right)
+\mathrm{dist}\left(\mathrm{pr}_t\left(P_{j,k}\right),
\mathrm{pr}_t\left(R^-(\gamma)\right)\right)\\
&<b_0\left[\frac{l(R)}{2^4}\right]^p+2(1+\gamma)[l(R)]^p+
2^4\Theta\left[\frac{l(R)}{2^4}\right]^p\\
&\leq\left[3+\frac{\Theta}{2^{4(p-1)}}+\frac{b_0}{2^4}\right][l(R)]^p
<b_0[l(R)]^p.
\end{align*}
Therefore, $(x,t)\in\Omega_{1,b_0}(R^-(\gamma))$, and hence
$\Omega_{1,b_0}(P_{j,k})\subset\Omega_{1,b_0}(R^-(\gamma))$. This, together with
the fact that $\{P_{j,k}:P\in\mathcal{D}_1(R^{++}(\gamma))\}$ are pairwise
disjoint [see \ref{(I)}] and Lemma \ref{finite overlapping}, further implies
that, for any given $j\in\mathbb{Z}_+\cap[0,2^4-1]$ and $k\in\mathbb{Z}_+$,
\begin{align*}
\sum_{P\in\mathcal{D}_1(R^{++}(\gamma))}
w\left(\Omega_{1,b_0}\left(P_{j,k}\right)\right)
&=\int_{\Omega_{1,b_0}(R^-(\gamma))}\sum_{P\in\mathcal{D}_1(R^{++}(\gamma))}
\boldsymbol{1}_{\Omega_{1,b_0}(P_{j,k})}w\\
&\leq B(n,1,b_0)w\left(\Omega_{1,b_0}\left(R^-(\gamma)\right)\right),
\end{align*}
and hence \eqref{20260804.1415} holds.

Next, we prove by induction that, for any $k\in\mathbb{Z}_+$, there exist
$D_k,\varepsilon_k\in[0,\infty)$, both independent of $\gamma$ and $R$,
such that
\begin{align}\label{20260804.1446}
w\left(R^-(\gamma)\right)\leq D_kw\left(R^-(\gamma)
+\left(\nu l(R)e_i,\Theta_k[l(R)]^p\right)\right)
+\varepsilon_kw\left(\Omega_{1,b_0}\left(R^-(\gamma)\right)\right).
\end{align}
Indeed, if $k=0$, then we can take $D_0:=0$ and $\varepsilon_0:=1$ since
$R^-(\gamma)\subset\Omega_{1,b_0}(R^-(\gamma))$ up to a set of measure zero.
We then assume that \eqref{20260804.1446} holds for some $k\in\mathbb{Z}_+$
and we show that \eqref{20260804.1446} holds for $k+1$. Let
$P\in\mathcal{D}_1(R^{++}(\gamma))$. From \ref{(II)}, we infer that there exist
$\alpha\in[0,\frac12]$ and $\widetilde{P}\in\mathcal{R}$ such that
$P=\widetilde{P}^-(\alpha)$. Applying this and \eqref{20260804.1446} with $R$
and $\gamma$ therein replaced, respectively, by $\widetilde{P}$ and $\alpha$,
we find that
\begin{align*}
w(P)\leq D_kw\left(P_{1,k}\right)+\varepsilon_kw\left(\Omega_{1,b_0}(P)\right).
\end{align*}
Using \eqref{20260804.1446} again with $R$ and $\gamma$ therein replaced,
respectively, by $\widetilde{P_{1,k}}$ and $\alpha$, we obtain
\begin{align*}
w(P)&\leq D_k\left[D_kw\left(P_{2,k}\right)
+\varepsilon_kw\left(\Omega_{1,b_0}\left(P_{1,k}\right)\right)\right]
+\varepsilon_kw\left(\Omega_{1,b_0}(P)\right)\\
&=D_k^2w\left(P_{2,k}\right)+\varepsilon_k\left[\sum_{j=0}^1D_k^j
w\left(\Omega_{1,b_0}\left(P_{j,k}\right)\right)\right].
\end{align*}
Proceeding in this way, we conclude that
\begin{align*}
w(P)\leq D_k^{2^4}w\left(P+\left(\nu l(R)e_i,
2^4\Theta_k\left[\frac{l(R)}{2^4}\right]^p\right)\right)
+\varepsilon_k\left[\sum_{j=0}^{2^4-1}D_k^j
w\left(\Omega_{1,b_0}\left(P_{j,k}\right)\right)\right].
\end{align*}
Summing over all $P\in\mathcal{D}_1(R^{++}(\gamma))$ and using
\eqref{20260804.1415} and the fact $\{P\}_{P\in\mathcal{D}_1(R^{++}(\gamma))}$
are pairwise disjoint, we find that
\begin{align}\label{20260804.1512}
w\left(R^{++}(\gamma)\right)
&=\sum_{P\in\mathcal{D}_1(R^{++}(\gamma))}w(P)\notag\\
&\leq\sum_{P\in\mathcal{D}_1(R^{++}(\gamma))}
\left\{D_k^{2^4}w\left(P+\left(\nu l(R)e_i,
2^4\Theta_k\left[\frac{l(R)}{2^4}\right]^p\right)\right)\right.\notag\\
&\quad\left.+\varepsilon_k\left[\sum_{j=0}^{2^4-1}D_k^j
w\left(\Omega_{1,b_0}\left(P_{j,k}\right)\right)\right]\right\}\notag\\
&\leq D_k^{2^4}w\left(R^{++}(\gamma)+\left(\nu l(R)e_i,
2^4\Theta_k\left[\frac{l(R)}{2^4}\right]^p\right)\right)\notag\\
&\quad+\varepsilon_kB(n,1,b_0)\left(\sum_{j=0}^{2^4-1}D_k^j\right)
w\left(R^-(\gamma)\right).
\end{align}
Moreover, note that
\begin{align*}
R^-(\gamma)+\left(\nu l(R)e_i,\Theta_{k+1}[l(R)]^p\right)
=R^{++}(\gamma)+\left(\nu l(R)e_i,
2^4\Theta_k\left[\frac{l(R)}{2^4}\right]^p\right)
+\left(\mathbf{0},(1-2\gamma)[l(R)]^p\right),
\end{align*}
\begin{align*}
R^-(\gamma)+\left(\nu l(R)e_i,\Theta_{k+1}[l(R)]^p\right)
=S_\nu^+(\gamma)+\left(\mathbf{0},\left[\Theta_{k+1}-(1+\gamma)\right]
[l(R)]^p\right),
\end{align*}
and $\max\{1-2\gamma,\,\Theta_{k+1}-(1+\gamma)\}\leq\Theta+3$. These, together
with Lemma \ref{Kim lemma 1 cor}(ii) with $\gamma^*$ and $\Theta$ therein
replaced, respectively, by $\frac12$ and $\Theta+3$, further implies that
\begin{align*}
&\max\left\{w\left(R^{++}(\gamma)+\left(\nu l(R)e_i,
2^4\Theta_k\left[\frac{l(R)}{2^4}\right]^p\right)\right),\,
w\left(S_\nu^+(\gamma)\right)\right\}\\
&\quad\leq C\left(n,p,q,\frac12,\Theta+3,[w]_{RH_q^+}\right)
w\left(R^-(\gamma)+\left(\nu l(R)e_i,\Theta_{k+1}[l(R)]^p\right)\right).
\end{align*}
Combining this, \eqref{20260804.1512}, and \eqref{20260803.2206},
we conclude that
\begin{align*}
w\left(R^-(\gamma)\right)
&\leq\left(C_1C_2+C_1^2\nu^\frac{1}{q'}D_k^{2^4}C_2\right)
w\left(R^-(\gamma)+\left(\nu l(R)e_i,\Theta_{k+1}[l(R)]^p\right)\right)\\
&\quad+C_1^2\nu^\frac{1}{q'}\varepsilon_kB(n,1,b_0)
\left(\sum_{j=0}^{2^4-1}D_k^j\right)
w\left(\Omega_{1,b_0}\left(R^-(\gamma)\right)\right),
\end{align*}
where $C_1:=C(n,p,q,\frac12,[w]_{RH_q^+})$ and
$C_2:= C(n,p,q,\frac12,\Theta+3,[w]_{RH_q^+})$. Thus, \eqref{20260804.1446}
holds for $k+1$ with $D_{k+1}:=C_1C_2+C_1^2\nu^\frac{1}{q'}D_k^{2^4}C_2$
and $\varepsilon_{k+1}:=C_1^2\nu^\frac{1}{q'}\varepsilon_kB(n,1,b_0)
(\sum_{j=0}^{2^4-1}D_k^j)$.

Finally, we determine $\nu\in(0,\frac14)$ and complete the proof of
Theorem \ref{shifting}. Specifically, choose $\nu\in(0,\frac14)$ such that
\begin{align}\label{20260804.1552}
C_1^2\nu^\frac{1}{q'}(2C_1C_2)^{2^4}C_2\leq C_1C_2\mbox{\ \ and\ \ }
C_1^2\nu^\frac{1}{q'}B(n,1,b_0)2^4\max\left\{1,\,2C_1C_2\right\}^{2^4-1}
\leq\frac12.
\end{align}
We prove that $D_k\leq2C_1C_2$ and $\varepsilon_k\leq\frac{1}{2^k}$
for any $k\in\mathbb{Z}_+$ by induction. Indeed, $D_0=0\leq2C_1C_2$
and $\varepsilon_0=1\leq\frac{1}{2^0}$. Suppose that $D_k\leq2C_1C_2$ and
$\varepsilon_k\leq\frac{1}{2^k}$ for some $k\in\mathbb{Z}_+$.
From the first inequality in \eqref{20260804.1552}, we deduce that
\begin{align*}
D_{k+1}=C_1C_2+C_1^2\nu^\frac{1}{q'}D_k^{2^4}C_2
\leq C_1C_2+C_1^2\nu^\frac{1}{q'}(2C_1C_2)^{2^4}C_2\leq2C_1C_2.
\end{align*}
Moreover, by the second inequality in \eqref{20260804.1552}, we obtain
\begin{align*}
\varepsilon_{k+1}&=C_1^2\nu^\frac{1}{q'}B(n,1,b_0)
\left(\sum_{j=0}^{2^4-1}D_k^j\right)\\
&\leq C_1^2\nu^\frac{1}{q'}B(n,1,b_0)2^4\max\left\{1,\,2C_1C_2\right\}^{2^4-1}
\leq\frac12\varepsilon_k\leq\frac{1}{2^{k+1}}.
\end{align*}
Thus, for any $k\in\mathbb{Z}_+$, $D_k\leq2C_1C_2$ and
$\varepsilon_k\leq\frac{1}{2^k}$. On the other hand, observe that
$\Theta_k\uparrow\Theta$ as $k\to\infty$. This, together with the case
$\gamma^*:=\frac12$ of Lemma \ref{Kim lemma 1 cor}(ii),
further implies that, for any $k\in\mathbb{Z}_+$,
\begin{align*}
w\left(R^-(\gamma)+\left(\nu l(R)e_i,\Theta_k[l(R)]^p\right)\right)
\leq C_3w\left(R^-(\gamma)+\left(\nu l(R)e_i,\Theta[l(R)]^p\right)\right),
\end{align*}
where $C_3:=C(n,p,q,\frac12,\Theta,[w]_{RH_q^+})$ is the same as in
Lemma \ref{Kim lemma 1 cor}(ii). Combining this, \eqref{20260804.1446},
and the proven conclusion that $D_k\leq2C_1C_2$ and
$\varepsilon_k\leq\frac{1}{2^k}$ for any $k\in\mathbb{Z}_+$,
we obtain, for any $k\in\mathbb{Z}_+$,
\begin{align*}
w\left(R^-(\gamma)\right)&\leq2C_1C_2w\left(R^-(\gamma)
+\left(\nu l(R)e_i,\Theta_k[l(R)]^p\right)\right)
+\frac{1}{2^k}w\left(\Omega_{1,b_0}\left(R^-(\gamma)\right)\right)\\
&\leq2C_1C_2C_3w\left(R^-(\gamma)
+\left(\nu l(R)e_i,\Theta[l(R)]^p\right)\right)
+\frac{1}{2^k}w\left(\Omega_{1,b_0}\left(R^-(\gamma)\right)\right).
\end{align*}
Taking the limit $k\to\infty$ and applying the fact that
$w(\Omega_{1,b_0}(R^-(\gamma)))<\infty$, we conclude that \eqref{20260804.1606}
holds with $\nu$ satisfying \eqref{20260804.1552},
$\Theta:=\frac{3}{1-\frac{1}{2^{4(p-1)}}}$, and
$C(n,p,q,[w]_{RH_q^+}):=2C_1C_2C_3$.
This completes the proof of Theorem \ref{shifting}.
\end{proof}

\begin{remark}
Let all the notation be as in Theorem \ref{shifting}. We point out that the
upper bound $\frac12$ for $\gamma$ in Theorem \ref{shifting} is not
essential. Indeed, for any fixed $\gamma^* \in (0, 1)$, the conclusion of
Theorem \ref{shifting} holds uniformly for all $\gamma\in[0,\gamma^*]$,
with the constants $v$, $\Theta$, and $C(n,p,q,[w]_{RH_q^+})$ additionally
depending on $\gamma^*$. This extension can be achieved via a straightforward
adaptation of the dyadic lattice construction of parabolic rectangles.
We omit the details.
\end{remark}

Now, we are ready to show Theorem \ref{question 4.6}.

\begin{proof}[Proof of Theorem \ref{question 4.6}]
From Lemma \ref{Kim lem 2}, we infer that, to prove Theorem
\ref{question 4.6}, it remains to show that, if $w\in RH_q^+$ for some
$q\in(1,\infty]$, then $w$ satisfies $\eqref{20260730.1813}$.
To do this, let $R:=R(x,t,L)\in\mathcal{R}$ with $(x,t)\in\mathbb{R}^{n+1}$
and $L\in(0,\infty)$. Note that
\begin{align}\label{20260804.1749}
\inf\mathrm{pr}_t\left(\frac12R^+(\gamma)\right)
-\sup\mathrm{pr}_t\left(R^-(\gamma)\right)
&=\left[t+\frac{1+\gamma}{2}L^p-\frac{1-\gamma}{2}\frac{L^p}{2^p}\right]
-\left(t-\gamma L^p\right)\notag\\
&=\left(\frac{1+3\gamma}{2}-\frac{1-\gamma}{2^{p+1}}\right)L^p>0.
\end{align}

Now, we cover $R^-(\gamma)$ by sufficiently small half parabolic rectangles.
Specifically, let $\nu$ and $\Theta$ be the same as in Theorem \ref{shifting}.
For any $N\in\mathbb{N}$, define $M_N:=n(\lceil\frac{N}{2\nu}\rceil+1)$. Then
\begin{align*}
\frac{M_N}{N^p}=\frac{n(\lceil\frac{N}{2\nu}\rceil+1)}{N^p}
\leq\frac{n}{2\nu N^{p-1}}+\frac{2n}{N^p}\to0\mbox{\ \ as\ \ }N\to\infty.
\end{align*}
Thus, we can fix $N:=N(n,p,q,\gamma,[w]_{RH_q^+})\in\mathbb{N}$ such that
\begin{align}\label{20260804.1754}
N\geq4,\ \ \frac{1}{N^p}\leq\frac{1-\gamma}{2^{p+1}},\mbox{\ \ and\ \ }
\frac{\Theta M_N}{N^p}\leq\frac12
\left(\frac{1+3\gamma}{2}-\frac{1-\gamma}{2^{p+1}}\right).
\end{align}
Let $l:=\frac{L}{N}$. Partition each spatial edge of $R^-(\gamma)$ into
$N$ equally long half-open intervals with length $l$. Then we obtain a sequence
$\{Q_j\}_{j\in\mathbb{N}\cap[1,N^n]}$ of subcubes of $Q(x,L)$. In addition,
$l(Q_j)=l$ for any $j\in\mathbb{N}\cap[1,N^n]$ and
$Q(x,L)=\bigcup_{j=1}^{N^n}Q_j$. On the other hand, by the second inequality in
\eqref{20260804.1754}, we obtain $l^p=\frac{L^p}{N^p}\leq
\frac{(1-\gamma)L^p}{2^{p+1}}<\frac{(1-\gamma)L^p}{2^p}<l_t(R^-(\gamma))$.
Therefore, we can cover the temporal edge of $R^-(\gamma)$ by
$K:=\lceil(1-\gamma)N^p\rceil$ half-open intervals of equal length $l^p$, with
bounded overlap 2. Then we obtain a sequence
$\{J_k\}_{k\in\mathbb{N}\cap[1,K]}$ of subintervals of the temporal
edge of $R^-(\gamma)$. Moreover, $|J_k|=l^p$ for any $k\in\mathbb{N}\cap[1,K]$.
For any $j\in\mathbb{N}\cap[1,N^n]$ and $k\in\mathbb{N}\cap[1,K]$, define
$P_{j,k}:=Q_j\times J_k$. Then there exists $R_{j,k}\in\mathcal{R}$ such that
$P_{j,k}=R_{j,k}^-(0)$. Moreover, $R^-(\gamma)=
\bigcup_{j=1}^{N^n}\bigcup_{k=1}^KP_{j,k}$ and $N^nK\leq2N^{n+p}$.

Next, we translate each $P_{j,k}$ into $\frac12R^+(\gamma)$. For any
$j\in\mathbb{N}\cap[1,N^n]$, denote the center of $Q_j$ by
$y_j:=(y_j^{(1)},\dots,y_j^{(n)})\in\mathbb{R}^n$ and assume that
$x=(x^{(1)},\dots,x^{(n)})$. For any $i\in\mathbb{N}\cap[1,n]$, choose
$m_j^{(i)}\in\mathbb{Z}$ such that
\begin{align}\label{20260804.1814}
\left|y_j^{(i)}+m_j^{(i)}\nu l-x^{(i)}\right|\leq\frac{\nu l}{2}.
\end{align}
This can be realized since \eqref{20260804.1814} is equivalent to
\begin{align*}
\frac{x^{(i)}-y_j^{(i)}}{\nu l}-\frac12\leq m_j^{(i)}\leq
\frac{x^{(i)}-y_j^{(i)}}{\nu l}+\frac12,
\end{align*}
which implies that $m_j^{(i)}$ belongs to a closed interval with length 1.
Moreover, for any $i\in\mathbb{N}\cap[1,n]$,
\begin{align*}
\left|m_j^{(i)}\right|\leq\frac{|x^{(i)}-y_j^{(i)}|}{\nu l}+\frac12
\leq\frac{L}{2\nu l}+\frac12=\frac{N}{2\nu}+\frac12.
\end{align*}
For any $j\in\mathbb{N}\cap[1,N^n]$, define $M_j:=\sum_{i=1}^n|m_j^{(i)}|$. Then
\begin{align}\label{20260804.1823}
M_j\leq\sum_{i=1}^n\left(\frac{N}{2\nu}+\frac12\right)
\leq n\left(\left\lceil\frac{N}{2\nu}\right\rceil+1\right)=M_N.
\end{align}
Also define the spatial translation vector $v_j:=\sum_{i=1}^nm_j^{(i)}\nu le_i$.
By sequentially translating along the $i$-th coordinate direction
$|m_j^{(i)}|$ times for any $i\in\mathbb{N}\cap[1,n]$, after $M_j$
translations in total, we can move $P_{j,k}$ to
\begin{align*}
\widetilde{P}_{j,k}:=P_{j,k}+\left(v_j,M_j\Theta l^p\right).
\end{align*}

We verify that, for any $j\in\mathbb{N}\cap[1,N^n]$ and
$k\in\mathbb{N}\cap[1,K]$,
\begin{align}\label{20260804.2054}
\mathrm{pr}_x\left(\widetilde{P}_{j,k}\right)\subset
Q\left(x,\frac{L}{2}\right)\mbox{\ \ and\ \ }
\sup\mathrm{pr}_t\left(\widetilde{P}_{j,k}\right)\leq
\inf\frac12R^+(\gamma).
\end{align}
Indeed, let $j\in\mathbb{N}\cap[1,N^n]$ and $k\in\mathbb{N}\cap[1,K]$.
Note that
$$\mathrm{pr}_x(\widetilde{P}_{j,k})=Q_j+v_j
=Q_j+\sum_{i=1}^nm_j^{(i)}\nu le_i.$$
From this, \eqref{20260804.1814},
$\nu<\frac14$, and the first inequality in \eqref{20260804.1754}, it
follows that, for any $y\in\mathrm{pr}_x(\widetilde{P}_{j,k})$,
\begin{align*}
\|y-x\|_\infty\leq\left\|y-\left(y_j+v_j\right)\right\|_\infty
+\left\|y_j+v_j-x\right\|_\infty\leq\frac{l}{2}+\frac{\nu l}{2}
=\frac{1+\nu}{2N}L<\frac14L.
\end{align*}
Thus, $\mathrm{pr}_x(\widetilde{P}_{j,k})\subset Q(x,\frac{L}{2})$.
Furthermore, $\mathrm{pr}_t(\widetilde{P}_{j,k})=\mathrm{pr}_t(P_{j,k})
+M_j\Theta l^p$. Using this, \eqref{20260804.1823}, the third inequality in
\eqref{20260804.1754}, and \eqref{20260804.1749}, we find that
\begin{align*}
\sup\mathrm{pr}_t\left(\widetilde{P}_{j,k}\right)
&\leq t-\gamma L^p+M_j\Theta l^p\leq t-\gamma L^p+\frac{\Theta M_N}{N^p}L^p\\
&\leq t-\gamma L^p+\frac12\left(\frac{1+3\gamma}{2}-
\frac{1-\gamma}{2^{p+1}}\right)L^p
<\inf\mathrm{pr}_t\left(\frac12R^+(\gamma)\right),
\end{align*}
and hence \eqref{20260804.2054} holds. We then translate $\widetilde{P}_{j,k}$
forward in time into $\frac12R^+(\gamma)$. To this end, we first use the
proven conclusion that $\frac{(1-\gamma)L^p}{2^p}=l_t(\frac12R^+(\gamma))$
to obtain
\begin{align}\label{20260804.2122}
\left[t+\frac{1+\gamma}{2}L^p-\frac{1-\gamma}{2}\frac{L^p}{2^p},
t+\frac{1+\gamma}{2}L^p-\frac{1-\gamma}{2}\frac{L^p}{2^p}+l^p\right)
\subset\mathrm{pr}_t\left(\frac12R^+(\gamma)\right).
\end{align}
In addition, from \eqref{20260804.2054}, we deduce that
$\inf\frac12R^+(\gamma)-\inf\mathrm{pr}_t(\widetilde{P}_{j,k})\geq l^p>0$.
This, together with the fact that $\inf\mathrm{pr}_t(\widetilde{P}_{j,k})\geq
t-L^p$, further implies that
\begin{align*}
1\leq\frac{\inf\frac12R^+(\gamma)-\inf\mathrm{pr}_t(\widetilde{P}_{j,k})}{l^p}
\leq\left(1+\frac{1+\gamma}{2}-\frac{1-\gamma}{2^{p+1}}\right)N^p.
\end{align*}
Combining this and Lemma \ref{Kim lemma 1 cor}(ii) with $\gamma^*$ and $\Theta$
therein replaced, respectively, by $\frac12$ and
$(1+\frac{1+\gamma}{2}-\frac{1-\gamma}{2^{p+1}})N^p$, we find that
\begin{align}\label{20260804.2134}
w\left(\widetilde{P}_{j,k}\right)
&\leq C\left(n,p,q,\frac12,\left(1+\frac{1+\gamma}{2}-
\frac{1-\gamma}{2^{p+1}}\right)N^p,[w]_{RH_q^+}\right)\notag\\
&\quad\times w\left(\widetilde{P}_{j,k}+\left(\mathbf{0},\inf\frac12R^+(\gamma)-
\inf\mathrm{pr}_t\left(\widetilde{P}_{j,k}\right)\right)\right).
\end{align}
Moreover, both the first inequality of \eqref{20260804.2054} and
\eqref{20260804.2122} yield
\begin{align*}
\widetilde{P}_{j,k}+\left(\mathbf{0},\inf\frac12R^+(\gamma)-
\inf\mathrm{pr}_t\left(\widetilde{P}_{j,k}\right)\right)
\subset\frac12R^+(\gamma).
\end{align*}
Therefore, we have moved each $P_{j,k}$ into a subrectangle
$\widetilde{P}_{j,k}+(\mathbf{0},\inf\frac12R^+(\gamma)-
\inf\mathrm{pr}_t(\widetilde{P}_{j,k}))$ of $\frac12R^+(\gamma)$.

Finally, by Theorem \ref{shifting}, \eqref{20260804.1823},
\eqref{20260804.2134}, the fact that
$\widetilde{P}_{j,k}+(\mathbf{0},\inf\frac12R^+(\gamma)-
\inf\mathrm{pr}_t(\widetilde{P}_{j,k}))\subset\frac12R^+(\gamma)$,
and the proven conclusion that $N^nK\leq2N^{n+p}$, we conclude that
\begin{align*}
w\left(R^-(\gamma)\right)&\leq\sum_{j=1}^{N^n}\sum_{k=1}^K
w\left(P_{j,k}\right)\leq\sum_{j=1}^{N^n}\sum_{k=1}^K
\left[C\left(n,p,q,[w]_{RH_q^+}\right)\right]^{M_j}
w\left(\widetilde{P}_{j,k}\right)\\
&\leq\left[C\left(n,p,q,[w]_{RH_q^+}\right)\right]^{M_N}
\sum_{j=1}^{N^n}\sum_{k=1}^Kw\left(\widetilde{P}_{j,k}\right)\\
&\leq2N^{n+p}\left[C\left(n,p,q,[w]_{RH_q^+}\right)\right]^{M_N}\\
&\quad\times C\left(n,p,q,\frac12,\left(1+\frac{1+\gamma}{2}-
\frac{1-\gamma}{2^{p+1}}\right)N^p,[w]_{RH_q^+}\right)
w\left(\frac12R^+(\gamma)\right).
\end{align*}
Thus, \eqref{20260730.1813} holds with
\begin{align*}
C_d\left(n,p,q,\gamma,[w]_{RH_q^+}\right)
&:=2N^{n+p}\left[C\left(n,p,q,[w]_{RH_q^+}\right)\right]^{M_N}\\
&\quad\times C\left(n,p,q,\frac12,\left(1+\frac{1+\gamma}{2}-
\frac{1-\gamma}{2^{p+1}}\right)N^p,[w]_{RH_q^+}\right).
\end{align*}
This completes the proof of Theorem \ref{question 4.6}.
\end{proof}

\section{Proof of Theorem \ref{question 4.5}}
\label{section3}

In this section, we show Theorem \ref{question 4.5}. To this end, we first
establish a parabolic John--Nirenberg inequality for functions in
$\mathrm{BMO}^+(\gamma)$, which plays a key role in the proof of Theorem
\ref{question 4.5}. The proof of this inequality relies on a new one-sided
stopping time argument, which differs from those employed in \cite[Theorem
3.1]{kmy(ma-2023)} and \cite[Lemma 1]{mt(jlms-1994)}.

\begin{lemma}\label{J-N for BMO+}
Let $0<\gamma<\alpha<1$. Then there exist positive constants
$A:=A(n,p,\gamma,\alpha)$ and $B:=B(n,p,\gamma,\alpha)$ such that, for any
$f\in\mathrm{BMO}^+(\gamma)$, $R\in\mathcal{R}$, and $\lambda\in(0,\infty)$,
\begin{align}\label{20260806.1546}
\left|R^-(\alpha)\cap\left\{f-f_{R^+(\gamma)}>\lambda
\right\}\right|\leq Ae^{-\frac{B\lambda}{\|f\|_{\mathrm{BMO}^+(\gamma)}}}
\left|R^-(\alpha)\right|.
\end{align}
\end{lemma}

\begin{proof}
Let $f\in\mathrm{BMO}^+(\gamma)$, $R\in\mathcal{R}$, and
$\lambda\in(0,\infty)$. Without loss of generality, we may assume that
$\|f\|_{\mathrm{BMO}^+(\gamma)}\in(0,\infty)$. We divide the proof of the
present lemma into the following four steps.

\emph{Step 1.} In this step, we decompose $R^-(\alpha)$ into a dyadic lattice
by adopting the rectangular decomposition strategy from \cite[Theorem
3.1]{kmy(ma-2023)}. To do this, let $m$ be the smallest integer such that
$\alpha+3\leq2^{pm+1}(\alpha-\gamma)$, i.e.,
\begin{align}\label{20260807.1632}
\frac1p\log_2\frac{\alpha+3}{2(\alpha-\gamma)}\leq m
<\frac1p\log_2\frac{\alpha+3}{2(\alpha-\gamma)}+1.
\end{align}
Let $\mathscr{D}_0(R^-(\alpha)):=\{R^-(\alpha)\}$. Divide each spatial edge of
$R^-(\alpha)$ into $2^m$ equally long half-open intervals and partition the
temporal edge of $R^-(\alpha)$ into $\lceil2^{pm}\rceil$ equally long half-open
intervals. Then we obtain $k_1:=2^{mn}\lceil2^{pm}\rceil$ half-open
subrectangles of $R^-(\alpha)$ and we denote the collection of these $k_1$
subrectangles by $\mathscr{D}_1(R^-(\alpha))$. For any
$i\in\mathbb{N}\cap[2,\infty)$, we define $\mathscr{D}_i(R^-(\alpha))$ by
induction. Suppose that we have already defined
$\mathscr{D}_{i-1}(R^-(\alpha))$. Then, for any
$P\in\mathscr{D}_{i-1}(R^-(\alpha))$, we divide each spatial edge of $P$
into $2^m$ equally long half-open intervals and partition the temporal edge of
$P$ into $j_2$ equally long half-open intervals, where
\begin{align*}
j_2:=\begin{cases}
\displaystyle
\lfloor2^{pm}\rfloor &\displaystyle\mbox{if\ }
\frac{l_t(P)}{\lfloor2^{pm}\rfloor}<\frac{(1-\alpha)[l(R)]^p}{2^{pmi}},\\
\displaystyle
\lceil2^{pm}\rceil &\displaystyle\mbox{if\ }
\frac{l_t(P)}{\lfloor2^{pm}\rfloor}\geq\frac{(1-\alpha)[l(R)]^p}{2^{pmi}}.
\end{cases}
\end{align*}
Then we obtain $k_i$ half-open subrectangles of $R^-(\alpha)$, and
$\mathscr{D}_i(R^-(\alpha))$ is defined to be the collection of these $k_i$
subrectangles, where $k_i\in\mathbb{N}\cap[2^{mni}\lfloor2^{pm}\rfloor^i,
2^{mni}\lceil2^{pm}\rceil^i]$. Let $\mathscr{D}(R^-(\alpha)):=
\bigcup_{i\in\mathbb{Z}_+}\mathscr{D}_i(R^-(\alpha))$.
We observe that the following assertions hold.
\begin{enumerate}
\item[\rm(i)] For any $i\in\mathbb{Z}_+$,
$\{P\}_{P\in\mathscr{D}_i(R^-(\alpha))}$ are pairwise disjoint and
$R^-(\alpha)=\bigcup_{P\in\mathscr{D}_i(R^-(\alpha))}P$. Moreover, for any
$U,V\in\mathscr{D}(R^-(\alpha))$, $U\cap V\in\{U,\,V,\,\emptyset\}$.

\item[\rm(ii)] For any $i\in\mathbb{Z}_+$ and $P\in\mathscr{D}_i(R^-(\alpha))$,
$l_x(P)=\frac{l(R)}{2^{mi}}$ and $\frac12\frac{(1-\alpha)[l(R)]^p}{2^{pmi}}
\leq l_t(P)\leq\frac{(1-\alpha)[l(R)]^p}{2^{pmi}}$ (which can be proved by
induction). Moreover, there exists a unique $\widetilde{P}\in\mathcal{R}$ such
that the bottoms of $P$ and $\widetilde{P}$ coincide and
\begin{align*}
\frac12(1-\alpha)\left[l\left(\widetilde{P}\right)\right]^p\leq l_t(P)
\leq(1-\alpha)\left[l\left(\widetilde{P}\right)\right]^p.
\end{align*}

\item[\rm(iii)] For any $i\in\mathbb{N}$ and $P\in\mathscr{D}_i(R^-(\alpha))$,
there exists a unique $\pi P\in\mathscr{D}_{i-1}(R^-(\alpha))$ such that
$P\subset\pi P$ and $\widetilde{P}\subset\widetilde{\pi P}^-(\gamma)$, which
can be deduced from the calculation that
\begin{align*}
l_t\left(\widetilde{P}\right)-l_t(P)&\leq\frac{2[l(R)]^p}{2^{pmi}}
-\frac12\frac{(1-\alpha)[l(R)]^p}{2^{pmi}}
=\frac12\frac{(\alpha+3)[l(R)]^p}{2^{pmi}}\\
&\leq\frac{(\alpha-\gamma)[l(R)]^p}{2^{pm(i-1)}}
\leq l_t\left(\widetilde{\pi P}^-(\gamma)\right)-l_t(\pi P).
\end{align*}

\item[\rm(iv)] For almost every $(x,t)\in R^-(\alpha)$ and $i\in\mathbb{Z}_+$,
there exists a unique $P_i\in\mathscr{D}_i(R^-(\alpha))$ such that $(x,t)\in
P_i$ for any $i\in\mathbb{N}$. Moreover, $P_0\supset P_1\supset
P_2\supset\cdots$, $\bigcap_{i\in\mathbb{Z}_+}P_i=(x,t)$, and
$f_{\widetilde{P_i}^+(\gamma)}\to f(x,t)$ as $i\to\infty$ (which follows from
the parabolic Lebesgue differentiation theorem; see, for instance, \cite[Lemma
2.3]{kmy(ma-2023)}).
\end{enumerate}

\emph{Step 2.} In this step, we show the following geometric observation.
Let $\mathcal{A}\subset\mathscr{D}(R^-(\alpha))$ satisfy that, for any
$U,V\in\mathcal{A}$, $U\cap V=\emptyset$. Then there exists a subfamily
$\mathcal{A}^*$ of $\mathcal{A}$ such that $\widetilde{U^*}^+(\gamma)\cap
\widetilde{V^*}^+(\gamma)=\emptyset$ for any $U^*,V^*\in\mathcal{A}^*$ and
\begin{align}\label{20260806.1750}
\sum_{P\in\mathcal{A}}|P|\leq\frac{4-\gamma}{1-\gamma}
\sum_{P^*\in\mathcal{A}^*}\left|\widetilde{P^*}^+(\gamma)\right|.
\end{align}
Indeed, we sort the rectangles in $\mathcal{A}$ in decreasing order of spatial
edge length and we denote this rearrangement of $\mathcal{A}$ by
$\{P_i\}_{i\in\mathbb{N}}$. We construct $\mathcal{A}^*$ by selecting
rectangles inductively. For any $i\in\mathbb{N}$, if, for any
$P\in\mathcal{A}^*$, $\widetilde{P_i}^+(\gamma)\cap\widetilde{P}^+(\gamma)=
\emptyset$, then $P_i$ is selected into $\mathcal{A}^*$. Otherwise, $P_i$ is
discarded. In this way, we obtain the desired subfamily $\mathcal{A}^*$.
Note that, for any $P\in\mathcal{A}\setminus\mathcal{A}^*$, there exists
$P^*\in\mathcal{A}^*$ such that $\widetilde{P}^+(\gamma)\cap
\widetilde{P^*}^+(\gamma)\neq\emptyset$ and $l_x(P)\leq l_x(P^*)$. This,
together with (i), further implies that $\mathrm{pr}_x(\widetilde{P}^+(\gamma))
=\mathrm{pr}_x(P)\subset\mathrm{pr}_x(P^*)
=\mathrm{pr}_x(\widetilde{P^*}^+(\gamma))$. Furthermore, since
$\widetilde{P}^+(\gamma)\cap\widetilde{P^*}^+(\gamma)\neq\emptyset$,
it follows that
\begin{align*}
\inf\mathrm{pr}_t(P)+2[l_x(P)]^p
&=\sup\mathrm{pr}_t\left(\widetilde{P}^+(\gamma)\right)
>\inf\mathrm{pr}_t\left(\widetilde{P^*}^+(\gamma)\right)\\
&=\inf\mathrm{pr}_t(P^*)+(1+\gamma)\left[l_x(P^*)\right]^p
\end{align*}
and
\begin{align*}
\inf\mathrm{pr}_t(P^*)+2\left[l_x(P^*)\right]^p
\sup\mathrm{pr}_t\left(\widetilde{P^*}^+(\gamma)\right)
&>\inf\mathrm{pr}_t\left(\widetilde{P}^+(\gamma)\right)\\
&=\inf\mathrm{pr}_t(P)+(1+\gamma)[l_x(P)]^p,
\end{align*}
and hence
\begin{align*}
\inf\mathrm{pr}_t(P^*)-(1-\gamma)\left[l_x(P^*)\right]^p
<\inf\mathrm{pr}_t(P)
<\inf\mathrm{pr}_t(P^*)+2\left[l_x(P^*)\right]^p.
\end{align*}
Combining this and the fact that $l_t(P)\leq(1-\alpha)[l_x(P)]^p$ [which can
be deduced directly from (ii)], we find that
\begin{align*}
P\subset\mathrm{pr}_x(P^*)\times
\left(\inf\mathrm{pr}_t(P^*)-(1-\gamma)\left[l_x(P^*)\right]^p,
\inf\mathrm{pr}_t(P^*)+3\left[l_x(P^*)\right]^p\right),
\end{align*}
and hence $|P|\leq\frac{4-\gamma}{1-\gamma}|\widetilde{P^*}^+(\gamma)|$.
Summing over all $P\in\mathcal{A}$ associated with $P^*$, and then over all
$P^*\in\mathcal{A}^*$, we obtain \eqref{20260806.1750}.

\emph{Step 3.} In this step, we prove the following one-sided stopping time
result. There exists a positive constant $C:=C(\gamma,\alpha)\in[1,\infty)$
satisfying that, for any $P\in\mathscr{D}(R^-(\alpha))$, there exists a
subfamily $\mathcal{S}(P)$ in $\mathscr{D}(R^-(\alpha))$ such that the following
statements hold.
\begin{Plist}
\item\label{P1} $\{U\}_{U\in\mathcal{S}(P)}$ are pairwise disjoint and, for any
$U\in\mathcal{S}(P)$, $U\subset P$.

\item\label{P2} $\sum_{U\in\mathcal{S}(P)}|U|\leq\frac12|P|$.

\item\label{P3} For almost every $(x,t)\in P\setminus
\bigcup_{U\in\mathcal{S}(P)}U$, $f(x,t)\leq
f_{\widetilde{P}^+(\gamma)}+C\|f\|_{\mathrm{BMO}^+(\gamma)}$.

\item\label{P4} For any $U\in\mathcal{S}(P)$, $f_{\widetilde{U}^+(\gamma)}\leq
f_{\widetilde{P}^+(\gamma)}+[C+2^{m(n+p)}]\|f\|_{\mathrm{BMO}^+(\gamma)}$.
\end{Plist}
Indeed, let $C:=C(\gamma,\alpha):=\frac{4(4-\gamma)}{1-\alpha}\in[1,\infty)$
and $\widetilde{\mathcal{S}}(P)$ be the collection of all
$U\in\mathscr{D}(R^-(\alpha))$ such that $U\subset P$ and
\begin{align}\label{20260806.2200}
f_{\widetilde{U}^+(\gamma)}>f_{\widetilde{P}^+(\gamma)}
+C\|f\|_{\mathrm{BMO}^+(\gamma)}.
\end{align}
Moreover, let $\mathcal{S}(P)$ be the collection of
all maximal rectangles in $\widetilde{\mathcal{S}}(P)$, i.e.,
$U\in\mathcal{S}$ if and only if $U\subset P$,
$f_{\widetilde{U}^+(\gamma)}>f_{\widetilde{P}^+(\gamma)}
+C\|f\|_{\mathrm{BMO}^+(\gamma)}$, and $f_{\widetilde{\pi U}^+(\gamma)}
\leq f_{\widetilde{P}^+(\gamma)}+C\|f\|_{\mathrm{BMO}^+(\gamma)}$.
By the maximality, we find that $\{U\}_{U\in\mathcal{S}(P)}$ are pairwise
disjoint, and hence \ref{P1} holds. From \eqref{20260806.1750} with
$\mathcal{A}:=\mathcal{S}(P)$ therein, we infer that there exists a subfamily
$[\mathcal{S}(P)]^*$ of $\mathcal{S}(P)$ such that $\widetilde{U^*}^+(\gamma)
\cap\widetilde{V^*}^+(\gamma)=\emptyset$ for any
$U^*,V^*\in[\mathcal{S}(P)]^*$ and
\begin{align*}
\sum_{U\in\mathcal{S}(P)}|U|\leq\frac{4-\gamma}{1-\gamma}
\sum_{U^*\in[\mathcal{S}(P)]^*}\left|\widetilde{U^*}^+(\gamma)\right|.
\end{align*}
This, together with \ref{P1}, \eqref{20260806.2200}, the fact that
$\widetilde{U}\subset\widetilde{P}^-(\gamma)$ for any $U\in\mathcal{S}(P)$
[which follows from (iii)], and (ii), further implies that
\begin{align*}
\sum_{U\in\mathcal{S}(P)}|U|&\leq\frac{4-\gamma}{C(1-\gamma)}
\frac{1}{\|f\|_{\mathrm{BMO}^+(\gamma)}}
\sum_{U^*\in[\mathcal{S}(P)]^*}\int_{\widetilde{U^*}^+(\gamma)}
\left[f-f_{\widetilde{P}^+(\gamma)}\right]_+\\
&\leq\frac{4-\gamma}{C(1-\gamma)}
\frac{1}{\|f\|_{\mathrm{BMO}^+(\gamma)}}\int_{\widetilde{P}^-(\gamma)}
\left[f-f_{\widetilde{P}^+(\gamma)}\right]_+\\
&\leq\frac{4-\gamma}{C(1-\gamma)}\left|\widetilde{P}^-(\gamma)\right|
\leq\frac{2(4-\gamma)}{C(1-\alpha)}|P|=\frac12|P|,
\end{align*}
and hence \ref{P2} holds.

Now, we show \ref{P3} and \ref{P4}. Let $(x,t)\in P\setminus
\bigcup_{U\in\mathcal{S}(P)}$. Applying both the definition of $\mathcal{S}(P)$
and (iv), we find that there exists a sequence
$\{P_i\}_{i\in\mathbb{Z_+}}$ in $\mathscr{D}(R^-(\alpha))$ such that
$P_0\supset P_1\supset P_2\supset\cdots$,
$\bigcap_{i\in\mathbb{Z}_+}P_i=\{(x,t)\}$, and
$f(x,t)=\lim_{i\to\infty}f_{\widetilde{P_i}^+(\gamma)}
\leq f_{\widetilde{P}^+(\gamma)}+C\|f\|_{\mathrm{BMO}^+(\gamma)}$. Thus,
\ref{P3} holds. To prove \ref{P4}, we first claim that, for any
$U\in\mathscr{D}(R^-(\alpha))\setminus\{R^-(\alpha)\}$,
\begin{align}\label{20260806.2242}
f_{\widetilde{U}^+(\gamma)}-f_{\widetilde{\pi U}^+(\gamma)}\leq2^{m(n+p)}
\|f\|_{\mathrm{BMO}^+(\gamma)}.
\end{align}
Indeed, from (iii), we deduce that
\begin{align*}
f_{\widetilde{U}^+(\gamma)}-f_{\widetilde{\pi U}^+(\gamma)}
&\leq\fint_{\widetilde{U}^+(\gamma)}
\left[f-f_{\widetilde{\pi U}^+(\gamma)}\right]_+\\
&\leq\frac{|\widetilde{\pi U}^-(\gamma)|}{|\widetilde{U}^+(\gamma)|}
\fint_{\widetilde{\pi U}^-(\gamma)}
\left[f-f_{\widetilde{\pi U}^+(\gamma)}\right]_+
=2^{m(n+p)}\|f\|_{\mathrm{BMO}^+(\gamma)},
\end{align*}
and hence the above claim holds. Let $U\in\mathcal{S}(P)$. Using
\eqref{20260806.2242} and the maximality of $U$, we obtain
\begin{align*}
f_{\widetilde{U}^+(\gamma)}\leq f_{\widetilde{\pi U}^+(\gamma)}
+2^{m(n+p)}\|f\|_{\mathrm{BMO}^+(\gamma)}
\leq f_{\widetilde{P}^+(\gamma)}+\left[C+2^{m(n+p)}\right]
\|f\|_{\mathrm{BMO}^+(\gamma)}.
\end{align*}
Therefore, \ref{P4} holds.

\emph{Step 4.} In this step, we show \eqref{20260806.1546} with the
help of \ref{P1}--\ref{P4} in Step 3. Define $\mathcal{G}_0:=\{R^-(\alpha)\}$
and, for any $k\in\mathbb{N}$, define $\mathcal{G}_k:=
\bigcup_{P\in\mathcal{G}_{k-1}}\mathcal{S}(P)$. Using \ref{P1} and \ref{P2}, we
conclude that, for any $k\in\mathbb{Z}_+$,
\begin{align}\label{20260807.1611}
\sum_{P\in\mathcal{G}_k}|P|\leq\frac{1}{2^k}\left|R^-(\alpha)\right|.
\end{align}
Furthermore, from \ref{P4}, \ref{P3}, and \ref{P1}, it follows that, for any
$k\in\mathbb{Z}_+$ and $P\in\mathcal{G}_k$,
\begin{align*}
f_{\widetilde{P}^+(\gamma)}\leq
f_{R^+(\gamma)}+k\left[C+2^{m(n+p)}\right]\|f\|_{\mathrm{BMO}^+(\gamma)}
\end{align*}
and, for almost every $(x,t)\in P\setminus\bigcup_{U\in\mathcal{G}_{k+1}}U$,
\begin{align*}
f(x,t)\leq f_{\widetilde{P}^+(\gamma)}+C\|f\|_{\mathrm{BMO}^+(\gamma)}
\leq f_{R^+(\gamma)}
+(k+1)\left[C+2^{m(n+p)}\right]\|f\|_{\mathrm{BMO}^+(\gamma)}.
\end{align*}
Therefore, for any $k\in\mathbb{Z}_+$,
\begin{align*}
R^-(\alpha)\cap\left\{f>f_{R^+(\gamma)}+k\left[C+2^{m(n+p)}\right]
\|f\|_{\mathrm{BMO}^+(\gamma)}\right\}\subset\bigcup_{P\in\mathcal{G}_k}P
\end{align*}
up to a set of measure zero. This, together with \eqref{20260807.1611} and
\ref{P1}, further implies that, for any $k\in\mathbb{Z}_+$,
\begin{align}\label{20260807.1616}
\left|R^-(\alpha)\cap\left\{f-f_{R^+(\gamma)}>k\left[C+2^{m(n+p)}\right]
\|f\|_{\mathrm{BMO}^+(\gamma)}\right\}\right|
\leq\sum_{P\in\mathcal{G}_k}|P|
\leq\frac{1}{2^k}\left|R^-(\alpha)\right|.
\end{align}

Finally, choose $k_\lambda\in\mathbb{Z}_+$ such that
\begin{align*}
k_\lambda\left[C+2^{m(n+p)}\right]\|f\|_{\mathrm{BMO}^+(\gamma)}\leq\lambda
<(k_\lambda+1)\left[C+2^{m(n+p)}\right]\|f\|_{\mathrm{BMO}^+(\gamma)}.
\end{align*}
Combining this, \eqref{20260807.1616} with $k:=k_\lambda$, and
\eqref{20260807.1632}, we obtain
\begin{align*}
\left|R^-(\alpha)\cap\left\{f-f_{R^+(\gamma)}>\lambda
\right\}\right|&\leq\frac{1}{2^{k_\lambda}}\left|R^-(\alpha)\right|
\leq2e^{-\frac{\ln2}{C+2^{m(n+p)}}
\frac{\lambda}{\|f\|_{\mathrm{BMO}^+(\gamma)}}}\left|R^-(\alpha)\right|\\
&\leq2e^{-\frac{\ln2}{\frac{4(4-\gamma)}{1-\alpha}
+2^{n+p}[\frac{\alpha+3}{\alpha-\gamma}]^\frac{n+p}{p}}
\frac{\lambda}{\|f\|_{\mathrm{BMO}^+(\gamma)}}}\left|R^-(\alpha)\right|.
\end{align*}
Thus, \eqref{20260806.1546} holds with $A:=2$ and
$B:=\frac{\ln2}{\frac{4(4-\gamma)}{1-\alpha}
+2^{n+p}[\frac{\alpha+3}{\alpha-\gamma}]^\frac{n+p}{p}}$. This completes the
proof of Lemma \ref{J-N for BMO+}.
\end{proof}

From Lemma \ref{J-N for BMO+}, we can directly deduce the following
exponential integrability for functions in $\mathrm{BMO}^+(\gamma)$.

\begin{corollary}\label{J-N for BMO+ cor}
Let $0<\gamma<\alpha<1$ and $f\in\mathrm{BMO}^+(\gamma)$. Then, for any
$\delta\in(0,\frac{B}{\|f\|_{\mathrm{BMO}^+(\gamma)}})$,
\begin{align*}
\sup_{R\in\mathcal{R}}\fint_{R^-(\alpha)}e^{\delta[f-f_{R^+(\gamma)}]_+}
\leq1+\frac{A\delta\|f\|_{\mathrm{BMO}^+(\gamma)}}
{B-\delta\|f\|_{\mathrm{BMO}^+(\gamma)}},
\end{align*}
where $A$ and $B$ are the same as in \eqref{20260806.1546}.
\end{corollary}

\begin{proof}
Let $f\in\mathrm{BMO}^+(\gamma)$, $R\in\mathcal{R}$, and $B$ be the same as in
\eqref{20260806.1546}. Fix $\delta\in(0,
\frac{B}{\|f\|_{\mathrm{BMO}^+(\gamma)}})$. By Cavalieri's
principle and Lemma \ref{J-N for BMO+}, we find that
\begin{align*}
\fint_{R^-(\alpha)}e^{\delta[f-f_{R^+(\gamma)}]_+}
&=\frac{1}{|R^-(\alpha)|}\int_0^\infty\left|R^-(\alpha)\cap
\left\{e^{\delta[f-f_{R^+(\gamma)}]_+}>\lambda\right\}\right|\,d\lambda\\
&\leq1+\frac{1}{|R^-(\alpha)|}\int_0^\infty e^\nu
\left|R^-(\alpha)\cap\left\{f-f_{R^+(\gamma)}
>\frac{\nu}{\delta}\right\}\right|\,d\nu\\
&\leq1+A\int_0^\infty e^{(1-\frac{B}{\delta\|f\|_{\mathrm{BMO}^+(\gamma)}})\nu}
\,d\nu\leq1+\frac{A\delta\|f\|_{\mathrm{BMO}^+(\gamma)}}
{B-\delta\|f\|_{\mathrm{BMO}^+(\gamma)}}.
\end{align*}
This completes the proof of Corollary \ref{J-N for BMO+ cor}.
\end{proof}

To prove Theorem \ref{question 4.5}, we need to invoke Theorem
\ref{question 4.6}. To this end, we first establish the following
time lag independence lemma for $RH_q^+$.

\begin{lemma}\label{Kim lem 3}
Let $q\in(1,\infty]$ and $0\leq\gamma<\alpha<1$, and let $w$ be a weight. If
there exists a positive constant $C$ such that, for any $R\in\mathcal{R}$,
\begin{align}\label{20260807.1928}
\left[\fint_{R^-(\alpha)}w^q\right]^\frac1q\leq C\fint_{R^+(\gamma)}w,
\end{align}
with the usual modification made when $q=\infty$,
then $w\in RH_q^+$ and $[w]_{RH_q^+}$ depends only on $n$, $p$, $q$, $\gamma$,
$\alpha$, and $C$.
\end{lemma}

\begin{proof}
Let $R\in\mathcal{R}$. From H\"older's inequality and \eqref{20260807.1928},
it follows that
\begin{align*}
\fint_{R^-(\alpha)}w\leq\left[\fint_{R^-(\alpha)}w^q\right]^\frac1q
\leq C\fint_{R^+(\gamma)}w,
\end{align*}
and hence
\begin{align}\label{20260807.1929}
w\left(R^-(\alpha)\right)\leq\frac{C(1-\alpha)}{1-\gamma}
w\left(R^+(\gamma)\right).
\end{align}
We claim that there exists a positive constant $D:=D(n,p,\gamma,\alpha,C)$
such that
\begin{align}\label{20260810.2058}
w\left(R^+(\gamma)\right)\leq Dw\left(R^+(\alpha)\right).
\end{align}

Applying this claim, \eqref{20260807.1928}, and Lemma \ref{Kim lemma 1}, we
conclude that $w\in RH_q^+(\alpha)$, and hence $w\in RH_q^+$ with
$[w]_{RH_q^+}$ depending only on $n$, $p$, $q$, $\gamma$, $\alpha$, and $C$.

Now, we show the above claim. Let $m\in\mathbb{N}$ be the smallest integer such
that $\frac{2}{m^p}<1-\alpha$, i.e.,
\begin{align}\label{20260807.2026}
\left(\frac{2}{1-\alpha}\right)^\frac1p<m\leq
\left(\frac{2}{1-\alpha}\right)^\frac1p+1.
\end{align}
Divide each spatial edge of $R^+(\gamma)\setminus R^+(\alpha)$ into $m$
half-open intervals of equal length $l:=\frac{l(R)}{m}$. Then we obtain a
sequence $\{Q_i\}_{i\in\mathbb{N}\cap[1,m^n]}$ of subcubes of
$\mathrm{pr}_x(R^+(\gamma)\setminus R^+(\alpha))$.
We also cover the temporal edge of $R^+(\gamma)\setminus R^+(\alpha)$ by
$J:=\lceil\frac{m^p(\alpha-\gamma)}{1-\alpha}\rceil$ equally long half-open
intervals of length $(1-\alpha)l^p$, with bounded overlap 2.
Then we obtain a sequence $\{J_j\}_{j\in\mathbb{N}\cap[1,J]}$ of subintervals
of $\mathrm{pr}_t(R^+(\gamma)\setminus R^+(\alpha))$. For any
$i\in\mathbb{N}\cap[1,m^n]$ and $j\in\mathbb{N}\cap[1,J]$, define
$S_{i,j}:=Q_i\times J_j$. Then there exists a unique $R_{i,j}\in\mathcal{R}$
such that $S_{i,j}=R_{i,j}^-(\alpha)$. From \eqref{20260807.1929},
we deduce that
\begin{align}\label{20260807.2000}
w\left(R_{i,j}^-(\alpha)\right)\leq\frac{C(1-\alpha)}{1-\gamma}
w\left(R_{i,j}^+(\gamma)\right).
\end{align}

To proceed, if $R_{i,j}^+(\gamma)\nsubset R^+(\alpha)$, then we cover the
temporal edge of $R_{i,j}^+(\gamma)\setminus R^+(\alpha)$ by equally long
half-open intervals of length $(1-\alpha)l^p$, with bounded overlap 2.
We require at most $\lceil\frac{1-\gamma}{1-\alpha}\rceil$ intervals
and we apply \eqref{20260807.1929} to each subrectangle of
$R_{i,j}^+(\gamma)\setminus R^+(\alpha)$. We iterate this process of rectangle
decomposition and translation until we send every $R_{i,j}^+(\gamma)$ into
$R^+(\alpha)$. Note that $(1-\alpha)l^p<\frac{(1-\alpha)[l(R)]^p}{2}$, which
ensures that our movement remains within $R^+(\alpha)$.
Since the distance of each forward in time shift is at least
$(1+\gamma)l^p$, the number of iterations is at most
$\lceil\frac{m^p(\alpha-\gamma)}{1+\gamma}\rceil+1$.
By using \eqref{20260807.2000} repeatedly, we obtain, for any
$i\in\mathbb{N}\cap[1,m^n]$ and $j\in\mathbb{N}\cap[1,J]$,
\begin{align*}
w\left(R_{i,j}^-(\alpha)\right)\leq
\left(\left\lceil\frac{1-\gamma}{1-\alpha}\right\rceil
\max\left\{1,\frac{C(1-\alpha)}{1-\gamma}\right\}\right)
^{\lceil\frac{m^p(\alpha-\gamma)}{1+\gamma}\rceil+1}
w\left(Q_i\times\mathrm{pr}_t\left(R^+(\alpha)\right)\right).
\end{align*}
Summing over $j\in\mathbb{N}\cap[1,J]$ and then over $i\in\mathbb{N}\cap[1,m^n]$
and using \eqref{20260807.2026}, we conclude that there exists a positive
constant $D:=D(n,p,\gamma,\alpha,C)$ such that \eqref{20260810.2058} and
hence the above claim holds. This completes the proof of Lemma \ref{Kim lem 3}.
\end{proof}

To prove Theorem \ref{question 4.5}, we still need to recall some known
results concerning the relationship between $\mathrm{BMO}^\pm(\gamma)$,
$\mathrm{PBMO}^-(\gamma)$, $A_r^+(\gamma)$ with $r\in(1,\infty)$, and
$A_\infty^+(\gamma)$. Lemma \ref{Kim lema 4}(i) is precisely
\cite[Lemma 7.4]{ks(apde-2016)}, Lemma \ref{Kim lema 4}(ii) is exactly
\cite[Proposition 4.4]{ks(na-2016)} (see also \cite[Corollary
4.17]{kyyz(cvpde-2025)}), and Lemma \ref{Kim lema 4}(iii) is precisely
\cite[Proposition 4.1]{ks(na-2016)}.

\begin{lemma}\label{Kim lema 4}
Let $r\in(1,\infty)$ and $\gamma\in(0,1)$. Then the following assertions hold.
\begin{enumerate}
\item[\rm(i)] $\mathrm{PBMO}^-(\gamma)=\{\lambda\ln w:
\lambda\in(0,\infty)\mbox{\ \ and\ \ }w\in A_r^+(\gamma)\}$. More precisely,
for any $f\in\mathrm{PBMO}^-(\gamma)$, there exists a positive constant
$\varepsilon:=\varepsilon(n,p,\gamma,r,\|f\|_{\mathrm{PBMO}^-(\gamma)})
\in(0,\infty)$ such that $e^{\varepsilon f}\in A_r^+(\gamma)$.
Conversely, for any $w\in A_r^+(\gamma)$, $\ln w\in\mathrm{PBMO}^-(\gamma)$
and there exists a positive constant $D:=D(n,p,r,\gamma,[w]_{A_r^+(\gamma)})$
such that $\|\ln w\|_{\mathrm{PBMO}^-(\gamma)}\leq D$.

\item[\rm(ii)] $\mathrm{PBMO}^-(\gamma)=\mathrm{BMO}^+(\gamma)\cap
[-\mathrm{BMO}^-(\gamma)]$. Moreover, for any $f\in L_\mathrm{loc}^1$,
\begin{align*}
\|f\|_{\mathrm{PBMO}^-(\gamma)}\sim\|f\|_{\mathrm{BMO}^+(\gamma)}+
\|-f\|_{\mathrm{BMO}^-(\gamma)},
\end{align*}
where the positive equivalence constants are independent of $f$.

\item[\rm(iii)] For any weight $w$, $w\in A_r^+(\gamma)$ if and only if
$w\in A_\infty^+(\gamma)$ and $w^{1-r'}\in A_\infty^-(\gamma)$. Moreover,
$[w]_{A_r^+(\gamma)}$, $[w]_{A_\infty^+(\gamma)}$, and
$[w^{1-r'}]_{A_\infty^-(\gamma)}$ only depend on each other,
$n$, $p$, $r$, and $\gamma$.
\end{enumerate}
\end{lemma}

Now, we are ready to show Theorem \ref{question 4.5}.

\begin{proof}[Proof of Theorem \ref{question 4.5}]
We first prove (i). Let $\alpha:=\frac{\gamma+1}{2}$, $A$ and $B$ be the same
as in \eqref{20260806.1546}, and $f\in\mathrm{BMO}^+(\gamma)$. By Corollary
\ref{J-N for BMO+ cor} and Jensen's inequality, we find that,
for any $R\in\mathcal{R}$,
\begin{align*}
\left[\fint_{R^-(\alpha)}e^{\frac{Bf}{2\|f\|_{\mathrm{BMO}^+(\gamma)}}}\right]
^\frac12&\leq e^{\frac{Bf_{R^+(\gamma)}}{4\|f\|_{\mathrm{BMO}^+(\gamma)}}}
\left\{\fint_{R^-(\alpha)}e^{\frac{B[f-f_{R^+(\gamma)}]_+}
{2\|f\|_{\mathrm{BMO}^+(\gamma)}}}\right\}^\frac12\\
&\leq(1+A)^\frac12e^{\frac{Bf_{R^+(\gamma)}}{4\|f\|_{\mathrm{BMO}^+(\gamma)}}}
\leq(1+A)^\frac12\fint_{R^+(\gamma)}
e^{\frac{Bf}{4\|f\|_{\mathrm{BMO}^+(\gamma)}}},
\end{align*}
which, together with Lemma \ref{Kim lem 3}, further implies that there exists
a positive constant $C_1:=C_1(n,p,\gamma)$ such that
$e^{\frac{Bf}{4\|f\|_{\mathrm{BMO}^+(\gamma)}}}\in RH_2^+$ and
$[e^{\frac{Bf}{4\|f\|_{\mathrm{BMO}^+(\gamma)}}}]_{RH_2^+}\leq C_1$.
Combining this and Theorem \ref{question 4.6} with its proof, we conclude that
there exist $r:=r(n,p,\gamma)$ and $C_2:=C_2(n,p,\gamma)$ satisfying
$e^{\frac{Bf}{4\|f\|_{\mathrm{BMO}^+(\gamma)}}}\in A_r^+(\gamma)$ and
$[e^{\frac{Bf}{4\|f\|_{\mathrm{BMO}^+(\gamma)}}}]_{A_r^+(\gamma)}\leq C_2$.
This, together with (i) and (ii) of Lemma \ref{Kim lema 4}, further implies that
$f\in\mathrm{PBMO}^-(\gamma)\subset[-\mathrm{BMO}^-(\gamma)]$ with
$\|f\|_{\mathrm{PBMO}^-(\gamma)}\lesssim\|f\|_{\mathrm{BMO}^+(\gamma)}$, and
hence $\mathrm{BMO}^+(\gamma)\subset[-\mathrm{BMO}^-(\gamma)]$. By reversing
the temporal axis, we obtain the reverse inclusion,
thereby completing the proof of (i).

Now, we show (ii). From Jensen's inequality, it follows that
$\bigcup_{q\in[1,\infty)}A_q^+(\gamma)\subset A_\infty^+(\gamma)$. We then
prove the converse inclusion. Let $w\in A_\infty^+(\gamma)$ and $f:=\ln w$.
Then, for any $R\in\mathcal{R}$,
\begin{align*}
\fint_{R^-(\gamma)}\left[f-f_{R^+(\gamma)}\right]_+
&\leq\fint_{R^-(\gamma)}e^{f-f_{R^+(\gamma)}}
=\fint_{R^-(\gamma)}w\exp\left\{\fint_{R^+(\gamma)}\ln\frac1w\right\}
\leq[w]_{A_\infty^+(\gamma)}.
\end{align*}
Thus, $f\in\mathrm{BMO}^+(\gamma)$. From this and (i), we deduce that
$-f\in\mathrm{BMO}^-(\gamma)$. By this and Lemma \ref{Kim lema 4}(i), we
conclude that there exists $\varepsilon:=\varepsilon(n,p,\gamma,
[w]_{A_\infty^+(\gamma)})\in(0,\infty)$ such that
\begin{align*}
e^{-\varepsilon f}=w^{-\varepsilon}\in A_2^-(\gamma)\subset A_\infty^-(\gamma).
\end{align*}
Let $q:=1+\frac{1}{\varepsilon}$. Then $-\varepsilon=1-q'$, and hence
$w^{1-q'}\in A_\infty^-(\gamma)$, which, together with the assumption that
$w\in A_\infty^+(\gamma)$ and Lemma \ref{Kim lema 4}(iii), further implies that
$w\in A_q^+(\gamma)$. This completes the
proof of (ii) and Theorem \ref{question 4.5}.
\end{proof}

Using Theorem \ref{question 4.5} and the parabolic John--Nirenberg inequality
for functions in $\mathrm{PBMO}^-(\gamma)$ established in \cite[Theorem
4.1]{kmy(ma-2023)}, we can improve the two time lags in Lemma
\ref{J-N for BMO+} into the same one, which in turn yields the following new
characterizations of $\mathrm{BMO}^+(\gamma)$.

\begin{corollary}\label{J-N for BMO+ 2}
Let $\gamma\in(0,1)$. Then the following statements are mutually equivalent.
\begin{enumerate}
\item[\rm(i)] $f\in\mathrm{BMO}^+(\gamma)$.

\item[\rm(ii)] $f\in L_\mathrm{loc}^1$ and there exist positive constants $A$
and $B$ such that, for any $R\in\mathcal{R}$ and $\lambda\in(0,\infty)$,
\begin{align}\label{2026080.2128}
\left|R^-(\gamma)\cap\left\{f-f_{R^+(\gamma)}>\lambda\right\}\right|
\leq Ae^{-B\lambda}\left|R^-(\gamma)\right|.
\end{align}

\item[\rm(iii)] $f\in L_\mathrm{loc}^1$ and there exist positive constants
$\delta$ and $D$ such that
\begin{align*}
\sup_{R\in\mathcal{R}}\fint_{R^-(\gamma)}e^{\delta[f-f_{R^+(\gamma)}]_+}\leq D.
\end{align*}
\end{enumerate}
\end{corollary}

\begin{proof}
The proof of the implication (ii)$\implies$(iii) is similar to the proof of
Corollary \ref{J-N for BMO+ cor}, and the proof of the implication
(iii)$\implies$(i) is trivial; we omit the details. We only show that
(i)$\implies$(ii). Let $f\in\mathrm{BMO}^+(\gamma)$. From Theorem
\ref{question 4.5} and the parabolic John--Nirenberg inequality for functions
in $\mathrm{PBMO}^-(\gamma)$ (see, for instance, \cite[Theorem
4.1]{kmy(ma-2023)}), we deduce that there exist
positive constants $\widetilde{A}:=\widetilde{A}(n,p,\gamma)$ and
$B:=B(n,p,\gamma)$ such that, for any $R\in\mathcal{R}$, there exists
$c_R\in\mathbb{R}$ satisfying, for any $\lambda\in(0,\infty)$,
\begin{align}\label{20260810.2151}
\left|R^-(\gamma)\cap\{(f-c_R)_+>\lambda\}\right|
\leq\widetilde{A}e^{-\frac{B\lambda}
{\|f\|_{\mathrm{BMO}^+(\gamma)}}}\left|R^-(\gamma)\right|
\end{align}
and
\begin{align}\label{20260810.2152}
\left|R^+(\gamma)\cap\{(f-c_R)_->\lambda\}\right|
\leq\widetilde{A}e^{-\frac{B\lambda}
{\|f\|_{\mathrm{BMO}^+(\gamma)}}}\left|R^+(\gamma)\right|.
\end{align}
Combining \eqref{20260810.2152} and Cavalieri's principle, we obtain
\begin{align*}
c_R-f_{R^+(\gamma)}&\leq\fint_{R^+(\gamma)}(f-c_R)_-
=\frac{1}{|R^+(\gamma)|}\int_0^\infty
\left|R^+(\gamma)\cap\{(f-c_R)_->\lambda\}\right|\,d\lambda\\
&\leq\widetilde{A}\int_0^\infty e^{-\frac{B\lambda}
{\|f\|_{\mathrm{BMO}^+(\gamma)}}}
\,d\lambda=\frac{\widetilde{A}}{B}\|f\|_{\mathrm{BMO}^+(\gamma)}.
\end{align*}
This, together with \eqref{20260810.2151}, further implies that, for any
$\lambda\in(\frac{\widetilde{A}}{B}\|f\|_{\mathrm{BMO}^+(\gamma)},\infty)$,
\begin{align*}
\left|R^-(\gamma)\cap\left\{f-f_{R^+(\gamma)}>\lambda\right\}\right|
&\leq\left|R^-(\gamma)\cap\left\{f-c_R>\lambda
-\frac{\widetilde{A}}{B}\|f\|_{\mathrm{BMO}^+(\gamma)}\right\}\right|\\
&\leq\widetilde{A}e^{\widetilde{A}}e^{-\frac{B\lambda}
{\|f\|_{\mathrm{BMO}^+(\gamma)}}}\left|R^-(\gamma)\right|.
\end{align*}
On the other hand, for any
$\lambda\in(0,\frac{\widetilde{A}}{B}\|f\|_{\mathrm{BMO}^+(\gamma)}]$,
\begin{align*}
\left|R^-(\gamma)\cap\left\{f-f_{R^+(\gamma)}>\lambda\right\}\right|
&\leq\left|R^-(\gamma)\right|\leq e^{\widetilde{A}}e^{-\frac{B\lambda}
{\|f\|_{\mathrm{BMO}^+(\gamma)}}}\left|R^-(\gamma)\right|.
\end{align*}
Therefore, \eqref{2026080.2128} holds with
$A:=(1+\widetilde{A})e^{\widetilde{A}}$. This completes the proof of
the implication (i)$\implies$(ii), and hence Corollary \ref{J-N for BMO+ 2}.
\end{proof}

Furthermore, applying Theorem \ref{question 4.5} and the fact that
$\mathrm{PBMO}^+(\gamma)$ is independent of time lag (see, for instance,
\cite[Corollary 4.2]{kmy(ma-2023)} and \cite[Theorem 2.5]{kyyz(cvpde-2025)}),
we find that $\mathrm{BMO}^+(\gamma)$ is also independent of time lag;
we omit the details.

\begin{corollary}\label{BMO+ time lag}
Let $\gamma,\alpha\in(0,1)$. Then $\mathrm{BMO}^+(\gamma)$ and
$\mathrm{BMO}^+(\alpha)$ coincide. Moreover, for any $f\in L_\mathrm{loc}^1$,
$\|f\|_{\mathrm{BMO}^+(\gamma)}\sim\|f\|_{\mathrm{BMO}^+(\alpha)}$, where the
positive equivalence constants are independent of $f$.
\end{corollary}

From Theorem \ref{question 4.5} and \cite[Theorem
3.1]{kyyz(cvpde-2025)}, we deduce that the null space of
$\mathrm{BMO}^+(\gamma)$ consists of non-decreasing functions which depend
only on the temporal variable; we omit the details.

\begin{corollary}\label{BMO+ null space}
Let $\gamma\in(0,1)$. Then $f\in\{g\in\mathrm{BMO}^+(\gamma):
\|g\|_{\mathrm{BMO}^+(\gamma)}=0\}$ if and only if there exists a
non-decreasing function $\widetilde{f}:\mathbb{R}\to\mathbb{R}$ such that, for
almost every $(x,t)\in\mathbb{R}^{n+1}$, $f(x,t)=\widetilde{f}(t)$.
\end{corollary}

Finally, since the logarithm $\ln u$ of any positive weak solution $u$ to the
equation \eqref{20260811.1443} belongs to $\mathrm{PBMO}^-(\gamma)$, it follows
that $\ln u$ satisfies the following John--Nirenberg inequality;
we omit the details.

\begin{corollary}\label{solutions J-N}
Let $u$ be a positive weak solution to \eqref{20260811.1443}. Then,
for any $\gamma\in(0,1)$, $\ln u\in\mathrm{BMO}^+(\gamma)$ and there exist
positive constants $A:=A(n,p,\gamma)$ and $B:=B(n,p,\gamma)$ such that,
for any $R\in\mathcal{R}$ and $\lambda\in(0,\infty)$,
\begin{align*}
\left|R^-(\gamma)\cap\left\{\ln u-(\ln u)_{R^+(\gamma)}>\lambda
\right\}\right|\leq Ae^{-\frac{B\lambda}{\|\ln u\|_{\mathrm{BMO}^+(\gamma)}}}
\left|R^-(\gamma)\right|.
\end{align*}
\end{corollary}

\bigskip

\noindent\textbf{Acknowledgements}\quad The authors acknowledge the use of AI
tools during the exploratory stage of this project. All mathematical arguments
and proofs in the final manuscript were checked and written by the authors.

\bigskip

\noindent Weiyi Kong, Dachun Yang and Wen Yuan

\medskip

\noindent Laboratory of Mathematics
and Complex Systems (Ministry of Education of China),
School of Mathematical Sciences, Institute for Advanced Study,
Beijing Normal University,
Beijing 100875,
The People's Republic of China

\smallskip

\noindent{\it E-mails:} \texttt{weiyikong@mail.bnu.edu.cn} (W. Kong)

\noindent\phantom{{\it E-mails:} }\texttt{dcyang@bnu.edu.cn} (D. Yang)

\noindent\phantom{{\it E-mails:} }\texttt{wenyuan@bnu.edu.cn} (W. Yuan)

\end{document}